\documentclass[reqno, 12pt]{amsart}
\usepackage{lineno,hyperref}
\usepackage{amsmath,amssymb,amsfonts,amsbsy}
\usepackage{graphicx}
\usepackage{psfrag}
\usepackage{enumitem}
\usepackage{color}

\newtheorem{thm}{Theorem}[section]
 
 \newtheorem{lem}[thm]{Lemma}

 \newtheorem{defn}[thm]{Definition}
 \newtheorem{rem}[thm]{Remark}

\def\IR{\mathbb{R}}

\def\IN{\mathbb{N}}
\def\IP{\mathbb{P}}
\def\la{\lambda}

\def\de{\delta}

\def\ol{\overline}

\def\pl{\partial}

\def\no{\nonumber}
\def\be{\begin{equation}}
\def\ee{\end{equation}}
\def\bge{\begin{eqnarray}}
\def\ege{\end{eqnarray}}
\def\bgee{\begin{eqnarray*}}
\def\egee{\end{eqnarray*}}

 \numberwithin{equation}{section}

\begin{document}

\textsf{}
\smallskip
\title[A non-local stochastic parabolic problem]
{Finite-time blow-up  of a non-local stochastic parabolic problem}

 \author{Nikos I. Kavallaris}
\address{
 Department of Mathematics, University of Chester, Thornton Science Park
Pool Lane, Ince, Chester  CH2 4NU, UK
}

\email{n.kavallaris@chester.ac.uk}

\author{Yubin Yan}
\address{
 Department of Mathematics, University of Chester, Thornton Science Park
Pool Lane, Ince, Chester CH2 4NU, UK
}

\email{y.yan@chester.ac.uk}

\begin{abstract}
The main aim of the current work is the study of the conditions under which (finite-time) blow-up  of a non-local  stochastic parabolic problem occurs. 
 We first establish the existence and  uniqueness of the local-in-time  weak solution for such  problem.
The first part of the manuscript deals with the investigation of the conditions which guarantee the occurrence of noise-induced blow-up. In the second part we first prove the  $C^{1}$-spatial regularity of the solution. Then,  based on this regularity result, and using  a strong positivity result  we derive, for first in the literature of SPDEs, a Hopf's type boundary value point lemma. The preceding results together with Kaplan's eigenfunction method are then employed to provide a (non-local) drift term induced blow-up result. In the last part of the paper, we present a method which provides an upper bound of the probability of (non-local) drift term induced blow-up.
\end{abstract}



\date{\today}
\maketitle
\section{Introduction}
In the current work we consider the following non-local stochastic parabolic problem
\bge
&&\frac{\pl u}{\pl t}=\Delta u+F(u)+\sigma(u)\; \partial_t W(x,t),\; (x,t)\in D_T:=D\times(0,T),\qquad\label{eq1}\\
&& u(x,t)=0,\quad (x,t)\in \Gamma_T:=\pl D\times (0,T),\;\label{eq2}\\
&& u(x,0)=\xi(x),\quad x\in D\label{eq3},
\ege
where $T>0$ denotes the maximum existence time and $D\subset \IR^d,\;d\geq 1,$  is a bounded domain with smooth boundary, whilst $\Delta$ denotes the Laplacian operator.    Here  the non-local reaction (drift) term $F(u)$ is defined by
\begin{equation} \label{F1a}
 F(u): = \frac{\lambda f(u)}{\Big ( \int_{D} f(u) \, dx \Big )^{q}}, \quad q>0,
\end{equation}
for some positive constant $\la$   and  $f(u)$ is a  locally Lipschitz and nonnegative function. The diffusion  term $\sigma(u)$ is also assumed to be nonegative and Lipschitz continuous. Furthermore $\partial_t W(x,t)$ denotes by convention the formal time derivative of the Wiener process $\{W(x,t),\; x\in D,\; t\geq 0\}$ in a complete probability space $(\Omega,\mathcal{F}, \IP)$ with filtration $\left(\mathcal{F}_t\right)_{t\in[0,T]}$ generated by $W.$ The initial value $\xi$ is a $\mathcal{F}_{0}$-measurable variable in some suitable spaces introduced later.

Let $H= L^{2} (D)$ endowed  with the norm $\| \cdot \|$. The solution $u(t)\equiv u(x,t;\omega)$ of \eqref{eq1}-\eqref{eq3} should be understood as an $H$-valued stochastic process  on $[0, T]$ for some $T>0$ any realization $\omega\in \Omega.$  Thus, questions like local-in-time  existence and uniqueness of a solution of \eqref{eq1}-\eqref{eq3} arise and can be tackled with the approach developed in \cite{dazab}. Questions regarding the  temporal and spatial regularity of  \eqref{eq1}-\eqref{eq3}, which  are very interesting issues,  are addressed in the current work. In particular, for proving the occurrence of finite-time (non-local) drift induced blow-up  we need at least a $C^{1}-$spatial regularity result, which was only quite recently obtained for the general  quasilinear SPDEs, \cite{debdehof}, and we also revive in the current work for our semilinear problem \eqref{eq1}-\eqref{eq3}, cf. Theorem \ref{t4}.

We are strongly motivated to study problem \eqref{eq1} -\eqref{eq3} since  this kind of non-local stochastic problem is associated with various industrial processes  (e.g.  Ohmic heating in food sterilization \cite{KS18,ktbn, l1,l2,tz} and shear banding formation in high strain metals \cite{bt1,bt2,k15}) as well as with biological processes (e.g. chemotaxis phenomenon \cite{ks,KS18, w}) and  statistical mechanics  approaches\cite{kn}, where the multiplicative noise term $\sigma(u)\; \partial_t W(x,t)$ represents the existence of external perturbations or a lack of knowledge of certain physical parameters.
The occurrence of multiplicative noise terms is natural when one considers noisy control systems, see \cite{bl}, and its importance  is well-known in physics and  biology. Many experimental or numerical observations of self-organized behavior or phase transitions arising out of such noises have been recorded in \cite{h00,ssg07}. For a detailed construction of a mathematical model of the form \eqref{eq1} -\eqref{eq3} arising in shear banding formation in metals the interested readers can see \cite{k15}, whilst a stochastic model arising in MEMS technology is formulated and investigated  in \cite{k16}.

\section{State of the art}
The current work mainly focuses on the phenomenon of finite-time blow-up,  which in the probabilistic sense might be associated with the expectation of the solution $u$ of \eqref{eq1}-\eqref{eq3}  becoming infinitely big in finite time. Such a singular behaviour is definitely very interesting from the mathematical point of view, however in many applications in engineering and biology it is also correlated with some destructive behaviour  of the associated mathematical models. Thus the investigation of the conditions under which  such finite-time blow-up occurs becomes vital. So in the current paper we try to provide a thorough study of this issue for the non-local model \eqref{eq1} -\eqref{eq3}. Before stating and proving our main results, let us review the main blow-up results available in literature. Fundamental results on the blow-up of stochastic reaction diffusion equations were first obtained by Chow ( \cite{cho1, cho2}) but only for the local version of problem \eqref{eq1} -\eqref{eq3}, i.e. when $q=0$ in \eqref{F1a}. Chow's method actually implies finite-time blow-up  in the mean $L^p-$sense for  $p>1,$  see also Definition \ref{mmk}.  Lv and Duan in \cite{ ld},  following an approach  similar to \cite{cho1, cho2} and  again for the local problem, provide a further insight on the impact of the noise term in the blow-up phenomenon by describing the competition between the nonlinear reaction term and noise term.  Moreover,   Foondun et al.\cite{fln}, by extending Chow's ideas,  proved the nonexistence of global solutions for the
Cauchy problem, i.e. when $D=\IR^d,$  even for a fractional Laplacian operator.  Dozzi and L\'{o}pez-Mimbela  in \cite{dlm10}, by using a somewhat different approach, they proved a finite time blow-up result  for the local problem and for a superlinear reaction term when $\sigma(u)=u.$  Besides, their method also provides an upper bound of the probability of blow-up.

Thorough research has been undertaken regarding the study of finite-time blow-up  of deterministic, i.e. when $\sigma\equiv 0,$  reaction diffusion equations since the seminal work of Fujita \cite{fuj66, fuj69}. In particular, regarding  the deterministic  non-local problem
\bge
&&\frac{\pl u}{\pl t}=\Delta u+\frac{\lambda f(u)}{\Big ( \int_{D} f(u) \, dx \Big )^{q}},\quad (x,t)\in D_T,\;q>0,\qquad\label{ndp1}\\
&& u(x,t)=0,\quad (x,t)\in \Gamma_T,\;\label{ndp2}\\
&& u(x,0)=\xi(x), x\in D, \label{ndp3}
\ege
the finite-time blow-up, i.e. the occurrence of $T<\infty$  such that $$\limsup_{t\to T} ||u(\cdot,t)||_{\infty}=\infty,$$
where $\| \cdot \|_{\infty}$ denotes  the norm in $L^{\infty}(D),$ has been investigated in detail in \cite{bt2, kavtza, KN07, ks,l1,l2,tz}.  More precisely, the authors in \cite{kavtza} proved, that  for any $0<q<1 $ the solution $u$ of \eqref{ndp1}- \eqref{ndp3} on a convex domain $D$  blows in finite time, either for big values of the control parameter $\la$ or for big enough initial data $\xi(x),$ provided that $f(s)$ is a positive, increasing, convex function for any $s \in \mathbb{R}$ satisfying also the following conditions
\bgee
[f^{1-q}(s)]''\geq 0\quad\mbox{for}\quad s\in \IR \quad\mbox{and}\quad\int_{b}^{\infty} \frac{ds}{f^{1-q}(s)}  < \infty, \quad\mbox{for any} \, b\in \IR.
\egee
However, to the best of our knowledge there are no any blow-up  results for the stochastic non-local problem \eqref{eq1}-\eqref{eq3}. Hence, the current paper initiates an investigation in this direction. Our main techniques stem from the theory of nonlinear PDEs; in particular for our investigation we basically use ideas introduced and developed in \cite{ cho1,cho2, dlm10, kavtza}.

The structure of the paper is as follows. In the first part of Section \ref{fs}, we establish the existence and uniqueness of  a local-in-time solution of  the stochastic problem  \eqref{eq1}-\eqref{eq3}. The second part of Section \ref{fs} deals with the analysis of noise (diffusion) term induced  blow-up. Section \ref{psk} focuses on the demonstration of the non-local reaction (drift) induced  blow-up. To this end we first derive the $C^{1}$-spatial regularity for the solution $u$, and then as a by-product we prove a Hopf's type  lemma for some specific stochastic problems. Notably, as far as we know it is the first time in the context  of SPDEs that this  key result is stated and proven.
Next, making use of Hopf's lemma we derive an estimate of the solution $u$ of  \eqref{eq1}-\eqref{eq3},  near the boundary $\pl D$ in conjuction with the moving plane method, \cite{bc,gnn,nst}, adjusted in the context of SPDEs. Then the  latter key estimate in conjuction with Kaplan's eigenfunction method, \cite{kaplan}, lead to the  proof of  the desired  reaction (drift) induced  blow-up, which is analogous to blow-up result of the deterministic problem conisidered in \cite{kavtza}. Finally, Section \ref{psk} concerns with the derivation of an upper estimate of the probability of blow-up for the special case $\sigma(u)=u,$  via the method introduced in \cite{dlm10}.

\section{ Noise term induced  blow-up} \label{fs}
In the current section we investigate the circumstances under which finite-time  blow-up of the system \eqref{eq1}-\eqref{eq3}  occurs due to the presence of the  noise (diffusion)  term.  We first consider the existence and uniqueness  of a weak solution locally in time by  using It\^o's formula and a semigroup approach. Local-in-time existence and uniqueness are rather standard and one can easilly apeals to the well known results \cite{dazab}. However, due to the non-local nature of \eqref{eq1}-\eqref{eq3} and  for the sake of completeness we  present a detailed proof.


\subsection{ Existence and uniqueness of a local-in-time  solution}
We first set up the main functional and stochastic framework which  will be used for our analysis throughout the manuscript.

Let  $A= - \Delta$ with $\mathcal{D}(A) = H_{0}^{1}(D) \cap H^{2}(D)$, where $H_{0}^{1}(D)$ and $H^{2}(D)$ denote the standard Hilbert spaces and  assume that $A$ has the eigenpairs $(\lambda_{j}, \phi_{j}), j=1, 2, 3, \dots.$  Let $L(H)$ denote the space of all bounded operators from $H= L^{2} (D)$ to $H$ and let $Q \in L(H)$ be a non-negative definite and symmetric bounded operator on $H$ with orthonormal eigenfunctions  $\chi_{j} \in H, j=1,2, 3, \dots$ and  corresponding eigenvalues $ \gamma_{j}> 0, j=1, 2, 3, \dots$ such that $\text{Tr} (Q) = \sum_{j=1}^{\infty} \gamma_{j} <\infty$. (i.e., $Q$ is of trace class).    For simplicity, hereafter we choose $\chi_{j}= \phi_{j}, j=1, 2, 3, \dots.$

Thereafter we let $W=W(x, t)$ denote the $H$-valued $Q$-Wiener process defined by
\begin{equation} \label{noise1}
W(x,t) = \sum_{j=1}^{\infty} \gamma_{j}^{1/2} \phi_{j}(x)  \beta_{j}(t), \quad\mbox{almost surely (a.s.)}\;,
\end{equation}
where $\beta_{j}(t)$ are independent and identically distributed  $\mathcal{F}_{t}$-Brownian motions.

For the trace  operator  $Q \in L(H)$ with $\text{Tr} (Q) <\infty$,  there exists a kernel $q(x, y)$ such that
\[
(Qu) (x) := \int_{D} q(x, y) u(y) \, dy, \quad\mbox{for any} \; x \in D, \; u \in H,
\]
see \cite[p. 42-43]{cho} and \cite[Definition 1.64]{lor}.
The kernel $q(x, y)$ is also called the covariance function of the $Q$-Wiener process $W(x,t)$.

Let $X$ be a  Banach space with the norm $\| \cdot \|_{X}$. then we define the following Hilbert space
$$
 L_{2}^{0}(H; X) =\left\{\psi \in  L(H,X): \;     \sum_{j=1}^{\infty}   \| \psi Q^{1/2} (\phi_{j})  \|_{X}^{2} = \sum_{j=1}^{\infty} \gamma_{j} \| \psi (\phi_{j})  \|_{X}^{2} <\infty \right\},
$$
with norm $\| \psi \|_{L_{2}^{0}} =  \Big ( \sum_{j=1}^{\infty} \gamma_{j} \| \psi (\phi_{j})  \|_{X}^{2} \Big )^{1/2};$
 here $L(H, X)$ denotes  the space of all bounded operators from $H$ to $X$.  Then for any functional  $\Psi: [0, T] \to L_{2}^{0} (H, X)$,  the stochastic integral $\int_{0}^{T} \Psi (t) \, d W(t)$ is well defined, see for example \cite{dazab}. For the sake of simplicity we drop the spatial dependence from $W(x,t)$ and hereafter we denote it by  $W(t).$

 In order to  write \eqref{eq1}-\eqref{eq3} in the abstract form, we define the Nemytskii operator
\be\label{F}
F: H \to H, \quad \mbox{with} \;  F(u) (x) :=\frac{f(u(x))}{\left(\int_D f(u(x))\,dx \right)^q}, \; \; \mbox{for any} \;\;x \in D\quad \mbox{and}\;\; q>0,
\ee
and for any $u\in H.$

Here $f: \IR\to \IR$  is assumed to be a  local  Lipschitz continuous function, that is,  for any $s_0\in \IR$ there exist $\delta>0$ and $C_f$ such that for any $s_1,s_2\in\{s\in \IR: |s-s_0|<\delta\}$ there holds
\be\label{cond0}
|f(s_1)-f(s_2)|\leq C_f |s_1-s_2|.
\ee
In addition, we  assume that
$
\sigma: H \to L_{2}^{0} (H, H)
$
is an $L_{2}^{0} (H, H)$-valued operator and    then we may write the  problem \eqref{eq1}-\eqref{eq3} as the following  It\^{o} equation in  $H,$
\begin{align}
& d u(t)   =  \left[-A  u(t) +F(u(t))\right]  \, dt + \sigma(u(t))\, dW(t), \quad 0<t<T, \label{intro1} \\
& u_0 = \xi,  \label{intro2}
\end{align}
 where $u(t):=u(\cdot,t).$

We now introduce the definition of  solutions of \eqref{intro1}-\eqref{intro2}, see  \cite{cho,lor}, which will be mainly used throughout the paper.
\begin{defn}   \label{weak}
A predictable $H$-valued stochastic process $\{ u(t): t \in [0, T] \}$ is called a {\it weak solution} of \eqref{intro1}- \eqref{intro2} if for any  $v \in \mathcal{D}(A)$ and almost every  (a.e.) $ t \in [0, T]$, the following equality holds
\be\label{wf}
(u(t), v)
=
(\xi, v) + \int_{0}^{t} \Big [
-(u(s), A v) + (F(u(s)), v) \Big ] \, ds
+
\int_{0}^{t} \big ( \sigma (u(s)) \, d W(s), v \big ),
\ee
\end{defn}
almost surely (a.s).
The weak  formulation \eqref{wf}  is chosen since it is  more appropriate for our study on finite-time blow-up.

It is also known, \cite{dazab,lor}, that any weak solution $u$ of \eqref{intro1}- \eqref{intro2} is also a {\it mild solution} of \eqref{intro1}- \eqref{intro2}, that is,  it satisfies the following equality in $H=L^2(D),$
\[
u(t) = E(t) \xi + \int_{0}^{t} E(t-s) F(u(s)) \, ds + \int_{0}^{t} E(t-s) \sigma ( u(s)) \, d W(s),
\]
where $E(t) = e^{-t A}$ is the analytic semigroup generated by $-A,$ see \cite{lor}. On the other hand, any regular enough mild solution is also a weak solution, cf. \cite{dazab,lor}.

Before proceeding with the local-in-time existence of   \eqref{intro1}-\eqref{intro2} we prove the following result which will be frequently used throughout this section.
\begin{lem} \label{lip}
Assume that $f$ satisfies condition \eqref{cond0}  and it is  also bounded below by a positive constant, i.e. $f(s)\geq m>0, \;s\in \IR.$
Then the operator $F$  defined by \eqref{F}, satisfies a locally Lipschitz condition. In particular, for any $u_0\in H$ there exist $\delta>0$ and $C_{F}>0$ such that for any $u_1,u_2\in B_{u_0,\delta}=\{u\in H: ||u-u_0||_{\infty}<\delta\}$ there holds
\be\label{F1}
||F(u_1)-F(u_2)||_{H}\leq C_{F}||u_1-u_2||_{H}.
\ee
\end{lem}

\begin{proof}
We have
\begin{eqnarray}
&& | F(u_1)(x) - F(u_{2})(x) |
= \left |
\lambda \frac{f(u_1(x))}{\left ( \int_{D} f(u_1(x)) \, dx \right )^{q}} -
\lambda \frac{f(u_{2}(x))}{\left ( \int_{D} f(u_{2}(x)) \, dx \right )^{q}}
\right | \nonumber \\
&& \leq
  \frac{ \lambda |f(u_1(x)) - f(u_{2}(x))|   }{\left ( \int_{D} f(u_1(x)) \, dx \right )^{q}} \nonumber \\
&& \quad  +
\frac{\lambda |f(u_{2}(x))|   }{\left ( \int_{D} f(u_1(x)) \, dx \right )^{q} \left ( \int_{D} f(u_{2}(x)) \, dx \right )^{q}}
\left |
 \left ( \int_{D} f(u_1(x)) \, dx \right )^{q}
-
 \left ( \int_{D} f(u_{2}(x)) \, dx \right )^{q}
\right |\nonumber \\
&&   \leq C_f (m |D|)^{-q}|u_1(x) - u_{2}(x)|\nonumber \\
&& \quad  +
\frac{\lambda |f(u_{2}(x))|   }{\left ( \int_{D} f(u_1(x)) \, dx \right )^{q} \left ( \int_{D} f(u_{2}(x)) \, dx \right )^{q}}
\left |
 \left ( \int_{D} f(u_1(x)) \, dx \right )^{q}
-
 \left ( \int_{D} f(u_{2}(x)) \, dx \right )^{q}
\right |.
\label{psk1}\quad\qquad
\end{eqnarray}
By the  mean value theorem  and taking also into account \eqref{cond0},  we obtain
\begin{eqnarray}
&&\left |  \left( \int_{D} f(u_1(x)) \, dx \right )^{q} - \left( \int_{D} f(u_{2}(x)) \, dx \right )^{q}  \right | \nonumber \\
&& = q \left (  \int_{D} f(\bar{u}(x)) \, dx \right )^{q-1} \left |
  \int_{D} f(u_1(x)) \, dx -  \int_{D} f(u_{2}(x)) \, dx
\right |  \nonumber \\
&& \leq q  \left ( \int_{D} f(\bar{u}(x)) \, dx \right)^{q-1}  \int_{D} \left |   f(u_1(x)) - f(u_{2}(x))   \right | \, dx \nonumber \\
&& \leq C_f\, q  \left( \int_{D} f(\bar{u}(x)) \, dx \right )^{q-1}  \int_{D} |u_1(x) - u_{2}(x)|    \, dx \nonumber\\
&& \leq \widehat{C}_f\int_{D} |u_1(x) - u_{2}(x)|    \, dx,  \label{psk2}
\end{eqnarray}
where $\bar{u}(x)$ is a  value between $u_1(x), u_2(x).$  Note that if $0<q<1$ then $\widehat{C}_f=C_fq (m |D|)^{q-1},$ otherwise if $q\geq 1$ then we take $\widehat{C}_f=C_fq (M |D|)^{q-1}$ where $M=\sup_{x\in D, u\in B_{u_0,\delta}}\{f(u(x))\}.$

Combining  \eqref{psk1} and \eqref{psk2}   we finally derive, by also using  H\"{o}lder's inequality,
\[
\|F(u_1) - F(u_{2}) \|_{H} \leq C_F \| u_1 - u_{2} \|_{H},\quad\mbox{whenever}\quad u_1,u_2 \in B_{u_0,\delta}.
\]
The proof of Lemma \ref{lip} is now complete.
\end{proof}
Next we establish the local-in-time existence of a  weak  solution  to \eqref{intro1}-\eqref{intro2}.

\begin{thm} {\bf(Local-in-time  existence)}\label{existence}
Assume that $ \xi$ is $\mathcal{F}_{0}$-measurable in $H$ with     $\xi\in L^2(\Omega;L^{\infty}(D))$  and \eqref{F1} holds. Assume also that $\sigma: H\to L_{2}^{0}(H, H)$ is a  locally Lipschitz continuous mapping, i.e. for any $u_0\in H$ there exist $\delta>0$ and $C_{\sigma}>0$ such that for any $u_1,u_2\in B_{u_0,\delta}=\{u\in H: ||u-u_0||_{\infty}<\delta\}$ there holds
\begin{equation} \label{new1}
\|\sigma(u_{1}) - \sigma(u_{2}) \|_{ L_{2}^{0}(H, H)} \leq C_{\sigma} \| u_{1} - u_{2} \|_{H}.
\end{equation}
 Then the following hold true:
\begin{enumerate}
\item
There exists $T>0$ such that \eqref{intro1}-\eqref{intro2} has a unique weak solution   $ u \in L^{2}((0, T);   L^{\infty}(D)\cap W^{1,2} (D))  \cap L^{\infty}((0, T);H)$.
\item
The solution $u$ admits $H$-valued continuous trajectories and satisfies
\begin{equation} \label{exsit_2}
\mathbb{E}\Big[ \sup_{0 \leq t \leq T} \| u(t) \|_{H}^2\Big]
+ \mathbb{E}\Big[ \int_{0}^{T} \| \nabla u(t) \|_{H}^{2} \, dt\Big]
\leq C\, \mathbb{E}\big[ \| \xi \|_{H}^2\big].
\end{equation}
\item
In particular  the solution $u,$ seen as a stochastic process, belongs to the following functional space
\bgee
 L^{2} \big ( \Omega; C \big ( [0, T];H \big ) \big )
\cap
 L^{2} \big ( \Omega; L^{2} \big ( (0, T);    L^{\infty}(D)\cap W^{1,2} (D) \big ) \big )
\cap
 L^{p} \big ( \Omega; L^{\infty}((0, T); L^{p}(D)) \big ),
\egee
for any $p \geq 2.$
\end{enumerate}
\end{thm}
\begin{rem}
It is worth noted that due to the regularity provided by Theorem \ref{existence}  any weak solution $u$ of \eqref{intro1}-\eqref{intro2} also satisfies 
for almost every  (a.e.) $ t \in [0, T]$ and almost surely (a.s.) the following 
\bgee
(u(t), v)
=
(\xi, v) + \int_{0}^{t} \Big [
-(\nabla u(s), \nabla v) + (F(u(s)), v) \Big ] \, ds
+
\int_{0}^{t} \big ( \sigma (u(s)) \, d W(s), v \big ),
\egee
for any  $v \in W^{1,2}(D),$ and  it  is also called a {\it  variational solution} of  \eqref{intro1}-\eqref{intro2}.  Notably, all the results in the present work hold also for varational solutions of \eqref{intro1}-\eqref{intro2}.
\end{rem}
A key tool for the proof  of Theorem \ref{existence}, is  the following version of It\^{o}'s Lemma in Hilbert spaces.
\begin{lem} [\cite{cho}]  \label{ito}
Assume that $F$ and $\sigma$ satisfy \eqref{F1} and \eqref{new1} respectively.
Assume further that $ \xi$ is $\mathcal{F}_{0}$-measurable in $H$ and that $u$ satisfies the It\^{o} process
\[
d u= \left(-A u + F(u) \right)\, dt + \sigma (u) \, dW(t), \quad u(0) = \xi.
\]
If $\psi: H \to \mathbb{R}$ is a $\;C^{2}(H, \mathbb{R})$ functional, then the following holds
\begin{align}
 d \psi (u(t))
& = \psi^{\prime} (u(t)) \left[ \left(
- A u(t)
+ F(u(t))\right)\, dt
+ \sigma ( u(t)) \, dW(t) \right] \notag \\
& \quad  + \frac{1}{2} \mbox{Tr} \Big ( \sigma^{*}(u(t)) \psi^{\prime \prime} (u(t)) \sigma (u(t))  \Big ) \, d W(t), \notag
\end{align}
$ \sigma^{*}$ denotes the dual (transpose) operator of the diffusion operator $\sigma.$
\end{lem}
\begin{proof}[ Proof of Theorem \ref{existence}]
The proof is inspired by  \cite[Theorem 3]{dms1}; actually in \cite{dms1}  the more general  quasilinear problem is tackled. In particular, here we apply the semigroup approach to establish the local-in-time  existence of a mild solution for the semilinear problem \eqref{intro1}-\eqref{intro2}, which is finally regular enough to be also a weak solution.

Denote
\[
S_{T}  = \left\{ u\big |\, u \in L^{2} \big (\Omega \times [0, T];   L^{\infty}(D)\cap H_{0}^{1} (D) \big )\right \},
\]
equipped  with  the norm, with some suitable $ \gamma>0, \delta>0$ determined later,
\[
\| u \|_{\gamma, \delta}^2
: = \mathbb{E}\Big [
\int_{0}^{T} e^{- \gamma t} \big (
\| u(t) \|_{H}^2 + \delta \| \nabla u(t) \|_{H}^2 \big ) \, dt
\Big ].
\]
It is clear that $\| \cdot \|_{\gamma, \delta}$ is equivalent to $ \| \cdot \|_{S_{T}} $ for any $u\in S_{T},$ where $$ \| u \|^2_{S_{T}}: = \mathbb{E}\left[ \int_{0}^{T} \| u(t) \|^2_{H_{0}^{1}(D)} \, dt\right].$$
Consider the map $\mathcal{M}: S_{T} \to L^{2} \big (\Omega \times [0, T];   L^{\infty}(D)\cap H_{0}^{1} (D) \big )$ which is defined by
\bge\label{mmk2}
\mathcal{M} (u) (t)
:= E(t)\, \xi
+ \int_{0}^{t} E(t-s) F(u(s)) \, ds
+ \int_{0}^{t} E(t-s) \sigma (u(s)) \, d W(s),\quad
\ege
where   $E(t)$ is the semigroup generated by $-A.$

In the following we shall employ the Banach's fixed point theorem to prove the existence and  uniqueness of a $u$ such that  $\mathcal{M} (u)=u$ in $S_T.$

{\it Step 1:}\rm\;  We first show that $\mathcal{M}: S_{T} \to S_{T}.$  To that end, we need to show that  for any $u \in S_{T}$,  $\mathcal{M} ( u) \in S_{T}$, i.e.,
\[
\| \mathcal{M} (u) \|_{S_{T}}^2 = \mathbb{E}\left[ \int_{0}^{T} \| \mathcal{M}(u) (t) \|_{H_{2}^{1}(D)}^2 \, dt\right] < \infty,
\]
which actually follows by the assumptions on $\xi, F$ and $\sigma$.

{\it Step 2:}\rm\; Next we  show that $\mathcal{M}$ is a contaction operator, i.e.  there exist positive constants $\gamma, \delta$ and $0<\kappa<1$ such that
\[
 \| \mathcal{M} (u) - \mathcal{M} (v) \|_{\gamma, \delta} \leq \kappa \| u - v \|_{\gamma, \delta},
\]
where $\kappa = \kappa (F, \sigma)$ depends on $F$ and $\sigma$.

In fact, by \eqref{mmk2} we have
\begin{align}
\bar{u}(t):
=
\mathcal{M} (u) (t)-\mathcal{M} ( v)(t)
=
& \int_{0}^{t} E(t-s) \Big ( F(u(s)) - F(v(s)) \Big ) \, ds  \notag \\
&   + \int_{0}^{t} E(t-s) \Big ( \sigma (u(s)) - \sigma (v(s)) \Big ) \, dW(s),  \notag
\end{align}
which satisfies the It\^{o} problem
\[
d \bar{u} + A \bar{u} \, dt = [ F(u)- F(v) ] \, dt + [ \sigma (u) - \sigma (v) ] \, d W(t).
\]
Let $ w(t)= \bar{u}(t)  e^{-\frac{\gamma}{2}  t},$ then $w(t)$ satisfies
\[
dw + A w\, dt =-  \frac{\gamma}{2} w \, dt    +  [ F(u)- F(v) ]  e^{- \frac{\gamma}{2} t} \, dt + [ \sigma (u) - \sigma (v) ] e^{- \frac{\gamma}{2} t}  \, d W(t).
\]
Implementing  It\^{o}'s formula, see Lemma \ref{ito}, with $\varphi (w) = \| w \|_{H}^2$
we deduce
\begin{align}
& e^{- \gamma T} \| \bar{u}(T)\|_{H}^2
+ 2 \int_{0}^{T} e^{-\gamma t} \| \nabla \bar{u}(t) \|_{H}^2 \, dt  \notag \\
& = - \gamma \int_{0}^{T} e^{-\gamma t} \| \bar{u}(t) \|_{H}^2 \, dt
+ 2 \int_{0}^{T} e^{-\gamma t}  \big ( \bar{u}(t), F(u(t)) - F(v(t)) \big ) \, dt \notag \\
& + \int_{0}^{T} e^{-\gamma t} \| \sigma (u(t)) - \sigma (v(t)) \|_{L_{2}^{0}(H, H)}^2 \, dt.  \label{existence10}
\end{align}
Notably  for any small $\epsilon>0$, and by virtue of Young's inequality, we obtain some constant $C_{\epsilon}$ depending on $\epsilon$ such that
\begin{align}
& 2  \int_{0}^{T} e^{-\gamma t} \big  ( \bar{u}(t), F(u(t)) - F(v(t)) \big ) \, dt  \notag \\
& \leq
\epsilon\int_{0}^{T} e^{-\gamma t} \| F(u(t)) - F(v(t)) \|_{H}^2 \, dt
+ C_{\epsilon} \int_{0}^{T} e^{-\gamma t} \| \bar{u}(t) \|_{H}^2 \, dt \notag \\
& \leq \epsilon\,\, C_F  \int_{0}^{T} e^{-\gamma t} \| u(t) - v(t) \|_{H}^2 \, dt
+ C_{\epsilon} \int_{0}^{T} e^{-\gamma t} \| \bar{u}(t) \|_{H}^2 \, dt,   \notag
\end{align}
taking also into account that $F$ satisfies a locally Lipschitz condition with constant $C_F$  by Lemma \ref{lip}.

Furthermore,  due to  \eqref{new1} we have
\[
\int_{0}^{T} e^{-\gamma t} \|\sigma (u(t)) - \sigma (v(t)) \|_{L_{2}^{0}(H, H)}^2  \, dt
\leq
C_{\sigma} \int_{0}^{T} e^{-\gamma t} \| u(t) - v(t) \|_{H}^2 \, dt
\]
and thus by virtue of \eqref{existence10} we obtain
\begin{align}\label{existence10a}
& e^{- \gamma T} \| \bar{u}(T)\|_{H}^2
+ 2 \int_{0}^{T} e^{-\gamma t} \| \nabla \bar{u}(t) \|_{H}^2 \, dt+ (\gamma- C_{\epsilon} )   \int_{0}^{T} e^{-\gamma t} \| \bar{u} (t) \|_{H}^2 \, dt  \notag \\
& \leq  (\epsilon\,\, C_F+ C_{\sigma})   \int_{0}^{T} e^{-\gamma t} \| u(t) - v(t) \|_{H}^2 \, dt. \end{align}
Taking the expectation on both sides of  \eqref{existence10a},  noting also that $\mathbb{E}\left[ e^{- \gamma T} \| \bar{u}(T)\|_{H}^2 \right] \geq 0$, we derive
\begin{align}
& (\gamma- C_{\epsilon} )  \mathbb{E}\left[ \int_{0}^{T} e^{-\gamma t} \| \bar{u} (t) \|_{H}^2 \, dt\right]
+ 2 \mathbb{E} \left[ \int_{0}^{T} e^{-\gamma t} \| \nabla \bar{u}(t) \|_{H}^2   \, dt\right]  \notag \\
& \leq \epsilon\,\, C_F   \mathbb{E}\left[ \int_{0}^{T} e^{-\gamma t} \| u(t) - v(t) \|_{H}^2 \, dt\right]
+ C_{\sigma} \mathbb{E} \left[ \int_{0}^{T} e^{-\gamma t} \| u(t) - v(t) \|_{H}^2 \, dt\right], \notag
\end{align}
or equivalently
\begin{align}
& \mathbb{E}  \left[\int_{0}^{T} e^{-\gamma t} \| \bar{u}(t) \|_{H}^2 \, dt\right]
+ \frac{2}{ \gamma - C_{\epsilon}}\,  \mathbb{E} \left[\int_{0}^{T} e^{-\gamma t} \| \nabla \bar{u} (t)  \|_{H}^2 \, dt\right] \notag \\
& \leq \frac{\epsilon\,\, C_F + C_{\sigma}}{\gamma - C_{\epsilon}} \mathbb{E} \left[\int_{0}^{T} e^{-\gamma t} \| u(t) - v(t) \|_{H}^2 \, dt\right], \notag
\end{align}
provided  $\gamma > C_{\epsilon}.$

Choosing  now $\gamma$  sufficiently large and suitable $\epsilon>0$  such that  $ 0 <  \frac{\epsilon\,\, C_F + C_{\sigma}}{\gamma - C_{\epsilon}}< \kappa <1$  we have
\begin{align}
& \mathbb{E} \left[\int_{0}^{T} e^{-\gamma t} \Big (
\| \bar{u}(t) \|_{H}^2 + \delta \| \nabla \bar{u}(t) \|_{H}^2 \Big ) \, dt\right]  \notag \\
& \leq  \kappa\,  \mathbb{E} \left[\int_{0}^{T} e^{-\gamma t}\Big ( \| u(t) - v(t) \|_{H}^2 + \delta \| \nabla (u(t) - v(t)) \|_{H}^2  \Big ) \, dt\right],  \notag
\end{align}
for  $\delta = \frac{2}{\gamma - C_{\epsilon}}.$ The latter entails
\[
\| \mathcal{M} ( u) - \mathcal{M} (v) \|_{\gamma, \delta} \leq \kappa \| u - v \|_{\gamma, \delta}, \quad 0 < \kappa <1,
\]
and thus by Banach's fixed point theorem, there exists a unique local-in-time solution  $ u \in S_{T}$ for the  problem  \eqref{intro1}-\eqref{intro2}.
Finally, the estimate \eqref{exsit_2} can be obtained by following a similar argument as in the proof of Theorem 3 in \cite{dms1}.

{\it Step 3:}\rm\;
Finally we show that  $u \in L^{p} \big ( \Omega; L^{\infty}((0, T); L^{p}(D)) \big ),\; p \geq 2.$  Note that $f: \mathbb{R} \to \mathbb{R}$ satisfies a local Lipschitz condition and hence  in conjunction with Lemma \ref{lip} we actually get
\[
\| F(u(t)) - F(v(t)) \|_{L^{p}(D)} \leq C_F \| u(t)- v(t) \|_{L^{p}(D)},
\]
and
\[
\| \sigma(u(t)) - \sigma(v(t)) \|_{L_{2}^{0}(H, L^{p}(D))} \leq C_{\sigma} \| u(t)- v(t) \|_{L^{p} (D)}.
\]
 Thus we may again  use the same arguments as in Steps 1 and 2 to  show
\[
 u \in L^{p} \big ( \Omega; L^{\infty}((0, T); L^{p}(D)) \big ),\quad\mbox{for}\quad p \geq 2,
\]
and so the mild solution  derived in Step 2 is also a weak solution, cf. \cite{dazab,lor}.
The proof of Theorem \ref{existence} is now complete.
\end{proof}
\begin{rem}\label{ys}
If we consider initial data $\xi(x)\geq 0$  almost surely (a.s) then our local solution $u(x,t)$ is also  positive a.s  by application of the comparison principle, see \cite{cpt}.
\end{rem}
\subsection{Noise term induced finite-time blow-up}\label{nbu}
In this subsection we investigate  the impact of the noise term on the phenomenon of finite-time blow-up. We actually prove that the finite-time blow-up occurs, when the noise term is so big that it dominates the drift term and thus leads the dynamics of the stochastic system.

Before proceeding further with the mathematical analysis, we  first  define the notion  of finite-time blow-up for problem \eqref{intro1}-\eqref{intro2}.
\begin{defn} {\bf  (Finite-time blow-up)}\label{mmk}\\
The solution $u$ of problem \eqref{intro1}-\eqref{intro2} (or equivalently that  of \eqref{eq1}-\eqref{eq3}) blows up in finite time in the  sense of  mean $L^p-$norm if there exists $0\leq T^*<\infty$ such that
$$
\limsup_{t\to T^*} \mathbb{E}\big[ \|u\|_{p}\big]=\infty,
$$
for some $1\leq p\leq \infty.$ Here $\| \cdot \|_{p}$ denotes the norm in $L^{p}(D)$.
\end{defn}

Throughout this subsection we assume  the following:\\
 $(S_1)$ The correlation function $q(x,y)$ of the Wiener process $W(t)$  is continuous and positive for any $x,y\in \overline{D}$ and satisfies
 \[
 \int_{D} \int_{D} q(x,y) w(x) w(y)\, dx dy\geq q_1 \int_{D} w^2(x)\, dx,
 \]
 for any positive $w\in H$ and for some $q_1>0.$   This actually means that the correlation function behaves as a steep Gaussian function.\\
 $(S_2)$ $\sigma(s)$ is  a convex function and there also exists a positive, strictly increasing, convex and superlinear function  $G_1(s)$  such that
 \bge\label{mmk3}
 \sigma^2(s)\geq 2\, G_1(s^2)\quad\mbox{for}\quad s\geq 0\quad\mbox{and}\quad \int_0^\infty \frac{ds}{G_1(s)}<\infty.
 \ege
 We also consider as  $(\lambda_{1}, \phi_{1})$  the first eigenpair of the operator $ A=-\Delta:\mathcal{D} ( A) =H_{0}^{1}(D) \cap W^{2,2}(D),$  i.e. there holds
\bgee
-\Delta\phi_1=\la_1\,\phi_1,\;x\in D\quad\mbox{and}\quad\phi_1=0,\;x\in \partial D.
\egee
It is known that  $\phi_{1},$ see \cite{d89},  has a constant sign on $D$ so we can take $\phi_{1}\geq 0$ on $\overline{D}$ and it can be also normalized so that
\be\label{nor}
\int_{D} \phi_1\,dx=1.
\ee
Next following the approach of \cite{cho1} we obtain the following.
 \begin{thm}\label{ntbua}
Under conditions $(S_1)$ and $(S_2)$  the (unique) local-in-time solution $u$ to \eqref{intro1}-\eqref{intro2} (or equivalently to \eqref{eq1}-\eqref{eq3}), and  provided  by Theorem \ref{existence},  blows up in finite time if the initial data  $\xi\in L^2(\Omega;H)$ satisfy  $\xi(x)\geq 0$ a.s.  and
\bgee
\theta(0)=\theta_0=:\mathbb{E}\left[\left(\int_{D} \xi(x)\, \phi_1(x)\,dx\right)^2\right]>\gamma,
\egee
where $\gamma$ is the largest root of the equation $\beta(s):=2\,\widehat{q}_1 G_1(s)-2\la_1\,s=0$
and $\widehat{q}_1$  is some positive constant.
 \end{thm}
 \begin{proof}
We first define
\bgee\label{exp}
\hat{u}(t): = \int_{D} u(x,t)\, \phi_{1}(x) \, dx.
\egee
Then we take $v=\phi_1$  as a test function into weak formulation \eqref{wf} to deduce
\begin{eqnarray}\label{bu1}
\hat{u}(t)
:= \int_{D} u (x,t)\, \phi_{1}(x) \, dx   &&=\int_{D}  \xi (x) \,\phi_{1}(x) \, dx
-\int_{0}^{t} \int_{D} u (x, s)\,  (A\phi_{1})(x)  \, dx \, ds  \no \\
&&   +\la\int_{0}^{t} \int_{D} \frac{f(u(x, s))\, \phi_{1}(x)}{\left(\int_{D} f(u(x, s))\,dx \right)^{q}} \, dx \, ds\no\\
&&+\int_{0}^{t} \int_{D}   \sigma(u(x,s))\,\phi_{1}(x)\,dx \, dW(x,s)\no\\
&&=\int_{D}  \xi (x)\,\phi_{1}(x) \, dx
-\la_1\int_{0}^{t} \int_{D} u(x, s)\,\phi_{1}(x) \, dx \, ds\no\\&& +\la\int_{0}^{t} \int_{D} \frac{f(u(x,s))\, \phi_{1}(x)}{\left(\int_{D} f(u(x,s))\,dx \right)^{q}} \, dx \, ds\no\\
&&+\int_{0}^{t} \int_{D}   \sigma(u(x,s))\,\phi_{1}(x) \,dx \, dW(x,s).\quad\qquad
\end{eqnarray}
Next  It\^{o}'s formula, i.e. Lemma \ref{ito}, for $\psi(u)=u^2$ implies
\bge\label{mk1a}
\hat{u}^2(t) &&=\left(\int_{D} \xi(x)\, \phi_1(x)\,dx\right)^2-2\la_1\int_0^t \hat{u}^2 (s)\,ds\nonumber\\
&&+2\,\lambda\int_0^t\int_{D} \hat{u} (s)\frac{f(u(x,s))\, \phi_{1}(x)}{\left(\int_{D} f(u(x,s))\,dx \right)^{q}}\,dx\,ds+ 2 \int_0^t\int_{D} \hat{u} (s)\sigma(u(x,s))\, \phi_{1}(x)\,dx\,\, dW(x,s)\no\\
&&+\int_0^t\int_{D}\int_{D} q(x,y) \phi_1(x)\,\phi_1(y) \sigma(u(x,s))  \sigma(u(y,s))\,dx\,dy\,ds,
\ege
where \eqref{bu1} is also taking into account.

Set $\theta(t):=\mathbb{E}\left[\hat{u}^2(t)\right],$ then by taking the expectation into \eqref{mk1a} and interchanging the order of expectation and integration by virtue of Fubini's theorem, we have
\bge\label{mk2a}
\theta(t)&&=\mathbb{E}\left[\left(\int_{D} \xi(x)\, \phi_1(x)\,dx\right)^2\right]-2\la_1\int_0^t \theta(s)\,ds\no\\
&&+2\mathbb{E}\left[\int_0^t\int_{D}\hat{u}(s)\frac{f(u(x,s))\, \phi_{1}(x)}{\left(\int_{D} f(u(x,s))\,dx \right)^{q}}\,dx\,ds\right]\no\\
&&+\mathbb{E}\left[\int_0^t\int_{D}\int_{D} q(x,y) \phi_1(x)\,\phi_1(y)  \sigma(u(x,s))  \sigma(u(y,s)) \,dx\,dy\,ds\right],\qquad
\ege
where we use the following result
\bgee
\mathbb{E}\left[ \int_0^t\int_{D} \hat{u} (s)\sigma(u(x,s))\, \phi_{1}(x)\,dx\,\, dW(x,s)\right]=0.
\egee
 Alternatively \eqref{mk2a} can be written in the differential form
\bge
\frac{d\theta}{dt}
= &&-2\la_1\,\theta(t)+2\,\lambda\mathbb{E}\left[\hat{u}(t)\int_{D}\frac{f(u(x,t))\, \phi_{1}(x)}{\left(\int_{D} f(u(x,t))\,dx \right)^{q}}\,dx \right]\no\\
&&+\mathbb{E}\left[\int_{D}\int_{D} q(x,y) \phi_1(x)\,\phi_1(y)  \sigma(u(x,t))  \sigma(u(y,t)) \,dx\,dy\right],\; t>0,\qquad\quad\label{mk3}
\ege
with initial condition $$\theta(0)=\theta_0:=\mathbb{E}\left[\left(\int_{D} \xi(x)\, \phi_1(x)\,dx\right)^2\right].$$
Now assumptions $(S_1)$ and $(S_2)$ along with  Jensen's and H\"{o}lder's inequalities imply that the third term in the right side of \eqref{mk3} is estimated as
\bgee
&&\mathbb{E}\left[\int_{D}\int_{D} q(x,y) \phi_1(x)\,\phi_1(y)  \sigma(u(x,t))  \sigma(u(y,t)) \,dx\,dy\right]\no\\
&&\geq q_1 \mathbb{E}\left[\int_{D} \phi_1^2(x) \sigma^2(u(x, t))\, dx\right]\geq \tilde{q}_1 \mathbb{E}\left[\int_{D} \phi_1(x) \sigma(u(x, t))\, dx\right]^2\no\\
&& \geq\tilde{q}_1 \mathbb{E}\left[\sigma^2(\hat{u}(t))\right]\geq 2\,\tilde{q}_1\mathbb{E}\left[G_1\left(\hat{u}^2(t)\right)\right]\geq 2\,\widehat{q}_1 G_1(\theta(t)),
\egee
for  some appropriate positive constant $\widehat{q}_1$,  where \eqref{nor} has been also taken into consideration.

Therefore $\theta$ satisfies
\bge
&&\frac{d\theta (t)}{dt}\geq-2\la_1\,\theta(t)+2\,\widehat{q}_1 G_1(\theta(t)):=\beta(\theta(t)),\; t>0, \label{mk4}\\
&& \theta(0)=\mathbb{E}\left[(\xi,\phi_1)^2 \right], \label{mk5}
\ege
using also the fact that the second term in \eqref{mk3} is positive, see also Remark \ref{ys}.

Let now $\gamma$ be the largest root of the equation $\beta(s)=0,$ then we have $\beta(s)>0$ for any $s>\gamma$ if $\gamma>0.$ Otherwise, if $\gamma=0$ then  we have $\beta(s)>0$ for any $s>0.$ Therefore, if we take $\theta(0)>\gamma$ then by \eqref{mk4}-\eqref{mk5} we have
\bgee
t\leq \int_{\theta_0}^{\theta(t)} \frac{ds}{\beta(s)} \leq \int_{\theta_0}^{\infty} \frac{ds}{\beta(s)}.
\egee
Next using that $G_1(s)$ is a superlinear function, due to \eqref{mmk3}, we derive
\bgee
t\leq \int_{\theta_0}^{\theta(t)} \frac{ds}{\beta(s)}\leq \int_{\theta_0}^{\infty} \frac{ds}{\beta(s)}\leq \frac{1}{N} \int_{\theta_0}^{\infty} \frac{ds}{G_1(s)}<\infty,
\egee
for some positive constant $N,$ hence 
\bge \label{mmk4}
\theta(t)\to \infty\quad\mbox{as}\quad t\to T_1,
\ege 
where
$$
T_1\leq \int_{\theta_0}^{\infty} \frac{ds}{G_1(s)}<\infty.
$$
Notably,  by virtue of H\"{o}lder's inequality we derive
\bge\label{mmka}
\theta(t):=\mathbb{E}\left[\hat{u}^2(t)\right]\leq \mathbb{E}\left[||u||^2_{2}\right],
\ege 
which in conjunction with \eqref{mmk4}  implies
\bgee
\mathbb{E}\left[||u||^2_{2}\right]\to \infty\quad\mbox{as}\quad t\to T^*\leq T_1.
\egee
The proof of the Theorem is now complete.
 \end{proof}
 \begin{rem}
Since $u$ is bounded in $D\times [0,T)$  then \eqref{nor} and  \eqref{mmka}, via Theorem \ref{existence},  imply
\[
\theta(t):=\mathbb{E}\left[\hat{u}^2(t)\right]\leq \mathbb{E}\left[||u||^2_{\infty}\right]\to \infty,\quad\mbox{as}\quad t\to T_b\leq T_1.
\]
 Consequently Theorem \ref{ntbua} entails the finite-time blow-up of the stochastic process $u$  with respect to  $L^{\infty}-$ norm as well as according to Definition \ref{mmk}. 
 \end{rem}
\begin{rem}
The result of Theorem \ref{ntbua} with $f(s)=e^s$ and $q> 1$ complements the results of Theorems 4.1 and 4.2 in \cite{bl1}. Indeed those theorems state that when $\sigma(s)=0,$ i.e. for the deterministic case, only a global-in-time solution exists. Consequently,  Theorem \ref{ntbua} unveils that a dominant noise can change dramatically the dynamical behaviour of the solution leading to finite-time blow-up. Moreover Theorem \ref{ntbua} ensures the occurrence of finite-time blow-up in the case $f(s)=e^s, q=1,$  for any dimension $d>2,$ a result that was only conjectured for the deterministic case and only proven for $d=2,$ see  in \cite{ks}. In the latter case problem \eqref{eq1}-\eqref{eq3} is stochastic perturbation of a problem which describes the biological phenomenon of chemotaxis and so the occurrence of finite-time blow-up describes the aggregation of a biological population.
\end{rem}
\section{Drift term induced blow-up}\label{psk}
This  section deals with the  finite-time blow-up of \eqref{eq1}-\eqref{eq3} induced by the non-local drift (reaction) term. For the proof of  such results a delicate estimate of the non-local term is needed, which  actually arises as a by-product of an estimate of the solution $u$ of \eqref{eq1}-\eqref{eq3} near the boundary $\partial D.$  The control of $u$ near the boundary is obtained via the moving plane method, which requires the validity of a Hopf's type result for the stochastic problem \eqref{eq1}-\eqref{eq3}. However, for such a result to be proven the  $C^{1}$-spatial regularity  of $u$ is necessary which is established below.

For the purposes of the current section the positive nonlinearity $f(s)$ is assumed to be  increasing and convex, i.e.
\be\label{cond1}
f'(s),\;f''(s)\geq 0\quad\mbox{for}\quad s\in \IR.
\ee
\subsection{Spatial regularity of the solutions of \eqref{eq1}-\eqref{eq3}}
In the sequel, by following the approach introduced in    \cite{debdehof}, we  prove the spatial  $C^{1}$-spatial regularity of the solutions of   \eqref{eq1}-\eqref{eq3}. Such a result  will be used to derive the desired control of the solution near the  boundary. Before  we proceed with the proof we introduce the required functional framework.

 Let $C^{\bar{\alpha}, \bar{\beta}} (\bar{D} \times [0, T] ), 0 < \bar{\alpha} \leq 1, \; 0 < \bar{\beta} \leq 1$ denote the H\"older spaces equipped with the norm

\[
\| g \|_{C^{\bar{\alpha}, \bar{\beta}}}
= \sup_{( x,t)} | g(x,t)|
+ \sup_{(x,t) \ne ( y, s)} \frac{| g( x,t) - g(y,s)|}{|x-y|^{\bar{\beta}}+|t-s|^{\bar{\alpha}} }.
\]
With usual modifications, we can also consider the case for $\bar{\alpha}, \bar{\beta} \geq 1$. Note that it holds
\[
C^{\bar{\alpha}} \Big ( [0, T]; C^{\bar{\beta}}(\overline{D}) \Big  )
\nsubseteqq
C^{\bar{\alpha}, \bar{\beta}}  \Big ( \overline{D} \times [0, T]  \Big  ),
\]
and therefore we have to distinguish these two spaces.

Let, for any $p >1, r \geq 0$,
\[
H^{r,p}(D)
=
\left\{ h\big |\, \| h \|_{H^{r,p}(D)}
:= \inf \{ \| g \|_{H^{r,p} (\mathbb{R}^{d})}, \; g |_{D} = h\}<\infty \right\},
\]
where, the so called Bessel potential space, is defined as
\[
H^{r,p} ( \mathbb{R}^{d})
= \Big \{ h\big |\, ( I - \Delta)^{r/2} h \in L^{p}( \mathbb{R}^{d}), \;  \Big  \},
\]
where
$$
(I - \Delta)^{r/2} h: = \mathcal{F}^{-1} \Big ( ( 1+ | \xi |^2)^{r/2} \hat{h} \Big ).
$$
Here  $ \hat{h}$ denotes the Fourier transform of $h$, i.e., $\hat{h}  = \mathcal{F} (h)$, and $\mathcal{F}^{-1}$ denotes the inverse Fourier transform.
The choice of this scale of function spaces is more natural for our method than the standard Sobolev spaces $W^{r,p} ( D),  p >1, r \geq 0,$  cf. \cite{debdehof}. The spaces  $H^{r,p} (D)$ are generally different from the Sobolev spaces $W^{r,p}(D).$ However, the two following cases can occur
\bgee
W^{r,p}(D)=H^{r,p}(D)\quad\mbox{if}\quad r\in \mathbb{N}_0, p\in[1,\infty)\quad\mbox{or}\quad r\geq 0,\;  p=2,
\egee
and
\bgee
H^{r+\varepsilon,p}(D) \hookrightarrow W^{r,p}(D) \hookrightarrow H^{r-\varepsilon,p}(D),\; r\in \mathbb{R}, \;p\in (1, \infty), \;\varepsilon>0.
\egee
Furthermore, if $D$ is sufficiently regular, as in our case, then $H^{r,p}(D)$ coincides with the space of  restrictions of functions in $H^{r,p}(\mathbb{R}^d)$ to $D$ and thus the Sobolev embedding theorem holds true. Then the spaces  $H_0^{r,p}(D), r\geq 0, p\in(1,\infty),$ are defined as the closure of $C_c^{\infty}(D)$ in $H^{r,p}(D).$ Note that $H_0^{r,p}(D)=H^{r,p}(D)$ whenever $r\leq 1/p,$ while  $H_0^{r,p}(D)$ is strictly contained in  $H^{r,p}(D)$  if  $r> 1/p.$

Finally, it is worth noting  that the Bessel potential spaces  $H^{r,p} (D), \, p \geq 2, r  > 0$   are well suited for the stochastic It\^o integration (see \cite{brz} for the precise construction of the stochastic integral).

Notably, in order to obtain the desired $C^{1}$-spatial regularity for the solution of \eqref{eq1}-\eqref{eq3}, we need some further restrictions on the diffusion operator $\sigma.$  Indeed, we consider  the following assumption:

 $ (\sigma) \quad $  $\sigma: H \to L_{2}^{0}(H, H^{r,p}(D))$ satisfies the linear growth condition, i.e.,
\[
\| \sigma (u) \|_{L_{2}^{0}(H, H^{r,p}(D))}
\leq
 C ( 1+ \| u \|_{H^{r,p}(D)} ), \quad \mbox{for} \;  p \geq 2, \quad\mbox{and}\quad  0 \leq r \leq 1.
\]
Then the following regularity result can be proved by using the approach demonstrated in \cite[Proposition 5.1]{debdehof}.  For readers' convenience and for the sake of completeness we provide below  a complete proof adjusted to the stochastic  problem  \eqref{intro1}-\eqref{intro2}.
\begin{thm} {\bf ($C^{1}$-spatial regularity )} \label{t4}
Let us consider that all assumptions of Theorem \ref{existence} hold true.
Further assume that the  condition $(\sigma)$ holds and that $f$ satisfies \eqref{cond0}. If $\,\xi\in L^m\left(\Omega;C^{1+l}(\overline{D})\right),$ for $\;m\geq 2,\;l>0,$ then for all $\alpha \in (0, 1/2)$, there exists $\beta >0$ such that
\be\label{reg}
u  \in L^{m} \left ( \Omega; C^{1+ \beta,\alpha} \left ( \overline{D} \times[0, T]   \right ) \right ), \quad\mbox{for any}\quad m\geq 2.
\ee
\end{thm}
\begin{proof}
 We first  show that, there exists $\eta >0$  such that
\bge\label{ps4a}
u \in L^{m} \Big ( \Omega; C^{\eta} ( \overline{D}\times [0, T]  ) \Big ), \quad\mbox{for any}\quad m \geq 2.
\ege
Set $u= y + z$, where $z$ solves the following linear SPDE 
\begin{align}
& dz  = -A z   \, dt + \sigma (u) \,\; d W(t),\; 0<t<T,\notag \\
&z(0) = 0, \notag
\end{align}
whilst $y$ is the unique solution of the linear deterministic PDE problem
\begin{align}
& \frac{dy}{dt}   = -A y +  F(u),\,\;  0<t<T,\notag \\
& y(0) = \xi. \notag
\end{align}
{\it Step 1.  H\"older regularity of $z$.}\rm\; By Theorem \ref{existence}, the weak solution  $u$ of \eqref{intro1}-\eqref{intro2} belongs to $L^{m}\Big( \Omega; L^{m}((0, T); L^{m}( D)) \Big ), m \geq 2.$  Then the assumption  $(\sigma)$ with $ r=0$   implies that $\sigma (u)$ belongs to $L^m\left(\Omega; L^m\left((0,T);  L_{2}^{0} (H, L^m (D) )\right)\right).$
Hence the H\"older's regularity for the stochastic integral
\bgee\label{ps0}
z= \int_{0}^{t} E(t-s) \sigma( u(s)) \, d W(s),
\egee
is easily obtained. Indeed, using the linear growth of $\sigma$ and the factorization method, see \cite[Corollary 3.5]{brz}, we have
\[
\mathbb{E}\left[ \| z \|^m_{C^{\gamma} ( [0, T]; H_0^{\delta,m}(D))}\right]
\leq C \Big (
1 + \mathbb{E} \| u \|^m_{L^{m}((0,T);L^m(D))} \Big ),
\]
where
$ \gamma \in [0, \frac{1}{2} - \frac{1}{m} - \frac{\delta}{2}), \; \delta \in (0, 1- \frac{2}{m}), m>2$.
Now assume that $m \geq 3$, then $\delta = \frac{1}{6}, \gamma = \frac{1}{12}$ satisfy the conditions above uniformly.
Choose $m \geq m_{0}= 7d$, where $ d$ is the spatial dimension and also take $\alpha = \delta - \frac{d}{m_{0}},$ then by Sobolev's embedding theorem, we have
\[
H^{\delta, m}(D)  \hookrightarrow C^{\alpha}(D),
\]
since $\delta - \frac{d}{m} > \delta - \frac{d}{m_{0}}= \alpha$.   Thus for any $m \geq m_{0}$,
\[
\mathbb{E}\left[\| z \|^m_{C^{\gamma} ([0, T]; C^{\alpha}(D))}\right]
\leq C \Big ( 1 + \mathbb{E}\left[ \| u \|^m_{L^{m}((0, T); L^{m}(D))} \right]\Big )< \infty.
\]
On the other hand, for $m \in [2, m_{0})$, we have
\[
\mathbb{E} \left[\| z \|^m_{C^{\gamma} ([0, T]; C^{\alpha}(D))}\right]
\leq \left (
\mathbb{E} \left[\| z \|_{C^{\gamma} ([0, T]; C^{\alpha}(D))}^{m_{0}}\right] \right )^{m/m_{0}} < \infty.
\]
Thus for any $m \geq 2$, we have
\be\label{ps5a}
\mathbb{E} \left[\| z \|_{C^{\gamma} ([0, T]; C^{\alpha}(D))}^{m} \right]< \infty.
\ee
{\it Step 2. H\"older regularity of $y$.}\rm\;
Due to Lemma \ref{lip}, the functional $F$  satisfies a locally Lipschitz condition  and hence the following estimate is valid
\bge\label{ik1a}
\mathbb{E} \left[\| F(u) \|_{L^{r}\left((0,T);L^{r}(D)\right)}^{r}\right] \leq C \left( 1+ \mathbb{E} \left[\| u \|_{L^{r}\left((0,T);L^{r}(D)\right)}^{r}\right] \right)<\infty,
\ege
for any $r \geq 2$  by virtue of Theorem \ref{existence}.

Now  choosing $ \; r \geq 2$ such that $\frac{2+ d}{r} < \frac{1}{2},$
we have by classical parabolic PDE theory (see Theorems 7.1 and 10.1 in \cite{ladsolura}),
\[
\| y \|_{C^{\alpha, \alpha/2}(\overline{D}\times [0, T])}
\leq C \Big (
1+ \| \xi\|_{C^{l}(\overline{D})} \Big )
\Big (
1 + \| F(u) \|_{L^{r}\left((0,T);L^{r}(D)\right)}^{2d +1} \Big ), \quad  r \geq 2,
\]
for  some $\alpha >0$ and thus
\[
\| y \|_{C^{\alpha, \alpha/2}(\overline{D}\times [0, T])}^{m}
\leq C \Big (
1+ \| \xi\|_{C^{l}(\overline{D})}^{2m}  \Big )
\Big (
1 + \| F(u) \|_{L^{r}\left((0,T);L^{r}(D)\right)}^{r} \Big ),
\]
provided that $ 2 (2d +1) m <r.$

Since $r$ is arbitrary in $[2, \infty)$, then \eqref{ik1a} implies that
\be\label{ps5}
\mathbb{E}\left[ \| y \|_{C^{\alpha, \alpha/2}(\overline{D}\times [0, T])}^{m}\right]
< \infty, \quad \mbox{for any}\quad \; m \in [2, \infty).
\ee
Choose now $ \eta = \min \{ \frac{\alpha}{2}, \gamma, \lambda \} >0 $, then taking into account \eqref{ps5a} and \eqref{ps5} we derive estimate \eqref{ps4a}. \\
{\it Step 3. Higher spatial H\"older regularity of $z$.}\rm\;
Given estimate \eqref{ps4a} and using also Sobolev's embedding theorem we conclude that  $u\in L^m\left(\Omega; L^m\left((0,T);H^{k,m}(D)\right)\right)$ for $k<\eta < 1/2,$ hence by the assumption $(\sigma)$, we have $$\sigma(u)\in L^m\left(\Omega; L^m\left((0,T);L_{2}^{0} (H, H^{k,m}(D))\right)\right).$$ Using again the  factorization method \cite[Corollary 3.5]{brz}, we obtain
\[
\mathbb{E}\left[ \| z \|^m_{C^{\gamma} ( [0, T]; H^{\delta+k,m}(D))}\right]
\leq C \Big (
1 + \mathbb{E} \left[\| u \|^m_{L^{m}((0,T);H^{k,m}(D))}\right] \Big ),
\]
 where $ \gamma \in [0, \frac{1}{2} - \frac{1}{m} - \frac{\delta}{2})$  and $\delta \in (0, 1- \frac{2}{m})$  for any  $m>2$.  In the sequel  we assume $m \geq m_{0}:=(d+4)/k$ and thus $\delta=1- 3/m_{0}$ and $\gamma = 1/(4 m_{0})$ satisfy the conditions above uniformly in $m \geq m_{0}.$   Notably we have that $(\delta+k)m>km\geq k m_0\geq d$ and thus the following Sobolev embedding holds true:
\[
H^{\delta+k,m}(D)  \hookrightarrow C^{\theta}(D), \;\mbox{for}\; \theta = k+ \delta - d/m_{0}.
\]
Moreover via the definition of $\delta$ we have  $\theta=k+1-\frac{d+3}{m_0}>1.$

Therefore, we finally deduce
\bgee
\mathbb{E}\left[ \| z \|^m_{C^{\gamma} ( [0, T]; C^{\theta}(D))}\right]
\leq C \Big (
1 + \mathbb{E} \left[\| u \|^m_{L^{m}((0,T);H^{k,m}(D))}\right] \Big ), \quad m \geq 2,
\egee
and for some $0< \gamma <1/2,$ that is
\be\label{ps7}
z\in L^m\left(\Omega;C^{\theta,\gamma} (\overline{D}\times [0, T]  )\right).
\ee
{\it Step 4. Higher spatial H\"older regularity of $y$.}\rm\;
Next, taking estimate \eqref{ps4a} as starting point and using Schauder's theory for deterministic parabolic PDEs \cite[Theorem 6.48]{lb} as well as the linear growth condition on non-local term $F$ we derive
\[
\| y \|_{C^{1+\alpha, (1+\alpha)/2} ( \overline{D}\times[0, T] )}^{m}
\leq C \Big (
1+ \| \xi\|_{C^{1+l}(\overline{D})} + \| F(u(t)) \|_{L^{r}\left((0,T);L^{r}\right(D))}^{r} \Big ),
\]
for $r\geq 2$ large enough. Hence
\be\label{ps8}
y\in L^m\left(\Omega;C^{1+\alpha, (1+\alpha)/2} ( \overline{D}\times[0, T] )\right), \; m \geq 2,
\ee
which combined with \eqref{ps7} implies
\bgee
u\in L^m\left(\Omega;C^{1+\beta_{1}, \gamma} (\overline{D}\times [0, T]  )\right),
\egee
with $\beta_{1} = \min \{ \theta -1, \alpha \}.$\\
{\it Step 5. Time regularity.}\rm\;  For any $ \gamma \in (0, 1/2)$,  due to \eqref{ps8}, it suffices to improve only the time regularity of $z.$
By following the same arguments employed in step 1 for the stochastic integral and using estimate \eqref{ps4a} we deduce
\[
\mathbb{E}\left[ \| z \|^m_{C^{\gamma} ( [0, T]; H^{1+k,m}(D))}\right]
<\infty,
\]
which, via the Sobolev embedding $H^{1+k,m}(D)  \hookrightarrow C^{1+\beta}(D),\; \beta<k,$ implies that
$$
z\in L^m\left(\Omega;C^{1+\beta_{1}, \gamma} (\overline{D}\times [0, T]  )\right), \; m \geq 2.
$$
Combining now the above estimate with \eqref{ps8} we obtain the desired regularity for $u$ and
the proof of Theorem \ref{t4} is complete.
\end{proof}
\begin{rem}
For the purposes of the current work the spatial regularity provided by Theorem \ref{t4} is sufficient. Nevertheless, under the assumption that the drift  term $F(u)$ is bounded, which is guaranteed by  \eqref{cond0} and \eqref{cond1}, we can get a higher spatial regularity for the solution $u$ of \eqref{intro1}-\eqref{intro2}. In particular, in that case for all $\alpha \in (0,1/2)$ there exists $\beta>0$ such that
\bge
u  \in L^{m} \left ( \Omega; C^{ 2+ \beta, \alpha} \left ( \overline{D}\times [0, T]   \right ) \right ), \;m\geq 2,
\ege
provided also that  $\xi\in L^2\left(\Omega;C^{2+l}(\overline{D})\right).$ Indeed, we can increase the spatial regularity of $u$ as long as we consider smoother initial data $\xi$ and smoother non-local terms $F.$ For more details see Propositions 5.2 and 5.3 in \cite{debdehof}.
\end{rem}

\subsection{Strong positivity and Hopf's lemma}\label{ntbu}

According to the approach introduced in \cite{kavtza}, the proof of the finite-time blow-up for the deteministic problem \eqref{ndp1}-\eqref{ndp3} requires a key estimate of the solution close to the spatial boundary, which is heavily based on Hopf's boundary  lemma. For proving  a  reaction (drift) term induced blow-up for the stochastic problem  \eqref{eq1}-\eqref{eq3}  we would like to adjust a similar approach with the deterministic case and thus  a  Hopf's type lemma  in the context of  SPDEs should be established.


For readers' convenience we first give a required definition as well as we recall Hopf's maximum principle for deterministic parabolic PDEs, see also \cite{fri, pw, smo}.
\begin{defn}(\cite{fri}) \label{issp}
Let $P_{0}=(x_0,t_0)$ be a point on the boundary of $D_T.$ If there exists a closed ball $B$ centered at $(\bar{x}, \bar{t})$ such that
\[
B \subset \ol{D}_T, \quad B \cap \partial D_T = \{ P_{0} \}, \; \bar{x} \ne x_{0},
\]
then we say that $P_{0}$ has the inside strong sphere condition.
\end{defn}
Note that the inside strong sphere condition automatically holds when $\Gamma_T$ is $C^2.$

 The following strong positivity result is a  key result for proving a Hopf's type lemma.
\begin{thm}(Strong positivity)\label{max}
 Let $V\equiv V(x, t;  \omega)$ be a weak  solution of the following stochastic problem
\bge
&&\frac{\pl V}{\pl t}=\Delta V+ \chi(V)+ \sigma(V) \partial_t W,\; (x,t)\in D_T,\label{yy1}\\
&& V(x,t)=0,\quad (x,t)\in \Gamma_T,\label{yy2}\\
&& V(x,0)=V_0(x),\quad x\in D.\label{yy3}
\ege
Let also $\sigma:H\to L_{2}^{0}(H, H)$ be a Lipschitz continuous function satisfying condition $(\sigma)$ with $\sigma(0)=0.$ Assume further that $\chi> 0$ with $\chi\in L^r((0,T);L^r(D))$ for some $r\geq 2.$  In addition we consider initial datum $V_0$ which is Holder continuous and satisfies $V_0>0$  a.s. in $D$ with $V_0=0$ on $\partial D. $  Then
\[
\IP\left\{V (x, t) > 0 \quad \mbox{in}\quad D\times[0,T]\right\}=1,
\]
that is 
\bgee
V>0\quad\mbox{a.s.  in}\quad  D\times[0,T].
\egee
\end{thm}
A proof of Theorem \ref{max} can be found either in \cite{mue}, see Theorem 6.13, or in \cite[Theorem 5.1.]{kho}.
\begin{rem}
Theorem \ref{max} actually implies that the solution $V$ of problem \eqref{yy1}-\eqref{yy3} attains its zero minimum along $\overline{D}_T$ only on the boundary $\Gamma_T,$  also due to the boundary condition \eqref{yy2}.
\end{rem}
Next we provide a Hopf's lemma for  semilinear parabolic SPDEs. In particular, the following holds:

\begin{thm} (Hopf's Lemma) \label{max2}
Let $V$ be a weak solution of the problem \eqref{yy1}-\eqref{yy3}
where again functions $\chi:H \to H,\;\sigma:H\to L_{2}^{0}(H,H)$ satisfy the same assumptions as in Theorem \ref{max}. Consider initial condition  $V_0$ which is Holder continuous and satisfies $V_0>0$  a.s. in $D$ with $V_0=0$ on $\partial D. $
Assume also that $\Gamma_T$ is smooth enough, e.g $\Gamma_T$ is $C^2,$ such that it has the inside strong sphere property,
then
\bge\label{fcov}
\frac{\partial V}{\partial \nu}\Big|_{P_{0}}<0,
\ege
for any $P_0\in \Gamma_T:=\pl D\times (0,T).$  Notably the notion of the derivative into \eqref{fcov}  should be understood in the classical sence since $V$  is $C^{1}$ with repsect to the spatial variable due to Theorem \ref{t4}.
\end{thm}
\begin{proof}
Let  $P_{0}=(x_0,t_0)\in \Gamma_T:=\partial D\times (0,T),$ then since $\Gamma_T$ is smooth enough so it has the inside strong sphere property, we can then construct a closed ball $B$ centered at  $(\bar{x}, \bar{t})\neq (x_0,t_0)$ and with radius $R$ such that
\[
B \subset \ol{D}_T, \quad B \cap \Gamma_T = \{ P_{0} \}, \; \bar{x} \ne x_{0},
\]
i.e., the ball $B$ is tangent to $\Gamma_T$ at the point $P_{0}.$ Without loss of generality we may assume that the interior of $B$ lies in $D_T \cap \mathcal{V}$ for some neighborhood $\mathcal{V}$ of $P_{0}.$ We also consider a ball $B_1$ centered at $P_0$  and of radius  $\rho<|x_0-\bar{x}|,$ see Fig.~\ref{fig1}.
\begin{figure}[b]
\begin{center}
\includegraphics[scale=2]{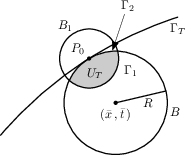}
\end{center}

\caption{The inside strong sphere condition configuration}
\label{fig1}       
\end{figure}

Let $\Gamma_1=\partial B_1\cap \overline{B}$ and $\Gamma_2=\partial B\cap B_1$ and let  $U_T$ be the region enclosed by the curves $\Gamma_1$ and $\Gamma_2.$ Since, by Theorem \ref{max}, $V>\min_{\overline{D}_T} V=0$  a.s. on $\Gamma_1$ then we can find $\eta>0$ such that
\begin{itemize}
\item $(i)$  $V\geq \eta>0\quad\mbox{on}\quad \Gamma_1,$  a.s.
\item $(ii)$ $V>0\quad\mbox{on}\quad \Gamma_2\setminus \{P_0\},$  a.s.  and
\item $(iii)$ $V(P_0)=0,$  a.s.,
\end{itemize}
see also Fig.~\ref{fig1}.

Consider now 
the auxiliary deterministic function  $h(x, t)$ defined by
\[
h(x, t) =  e^{-\alpha R^2}-e^{-\alpha \big ( | x- \bar{x}|^2 + (t- \bar{t})^2 \big )}<0,\quad \alpha>0,\quad\mbox{for}\quad (x,t)\in U_T,
\]
Evidently $h=0$ on $\pl B,$ and by selecting $\alpha$ sufficiently large we can attain $\mathcal{H}(h)=\frac{\partial h}{\partial t}-\Delta h:=g(x,t)<0$ in $U_T,$  see also \cite{smo}. There also holds 
\bge\label{mmks1}
\frac{\partial h}{\partial \nu}\Big|_{P_{0}}=2 \alpha R e^{-\alpha R^2}>0.
\ege
Let now $\Theta:= V + \varepsilon h,\;\varepsilon>0,$ 
then, in view of $(i),$ we can find $\varepsilon$ small enough such that $\Theta>0$ on $\Gamma_1$  a.s. .  Furthermore, by virtue of $(ii)$ and $(iii),$ along with the fact that $h=0$ on $\pl B,$ we derive $\Theta>0$  a.s. on $\Gamma\setminus\{P_0\}$ and $\Theta(P_0)=0$  a.s. . Note also that  $\Theta$ is a weak solution of
\bgee
&&\frac{\pl \Theta}{\pl t}=\Delta \Theta+\widetilde{\chi}(\Theta) +\sigma(\Theta-\varepsilon h) \partial_t W,\quad\mbox{in}\quad U_T,\\
&& \Theta(x,t)=0,\quad (x,t)\in \partial B,\;\\
&& \Theta(x,0)=\Theta_0(x),\quad x\in U_0,
\egee
where $\widetilde{\chi}(\Theta):=\chi(\Theta-\varepsilon h)+\varepsilon g(x,t)>0$  in $U_T$ and  $\Theta_0(x)>0$  in $U_0$ for choosing $\varepsilon$  small enough.

Accordingly, by virtue of Theorem \ref{max} we deduce that the minimum of $\Theta$ in $\overline{U}_T$ is attained only at $P_0.$

Therefore, since by Theorem \ref{t4} we have that $\Theta$ is $C^{1}$ with respect to the spatial variable on the boundary of $U_T,$  so  we finally  deduce
\[
\frac{\partial \Theta}{\partial \nu}\Big|_{P_{0}} \leq 0,\quad\mbox{a.s.},
\]
or equivalently
\bge\label{mkt1}
\frac{\partial V}{\partial \nu}+\varepsilon \frac{\partial h}{\partial \nu}     \leq 0,\quad\mbox{a.s.}, \quad \mbox{at} \; P_{0}.
\ege

Therefore, \eqref{mkt1} in  conjunction with  \eqref{mmks1} entails
\[
\frac{\partial V}{\partial \nu}\Big|_{P_{0}} <0,\quad\mbox{a.s.},
\]
and the proof  is complete.
\end{proof}
\begin{rem}
The result of Theorem \ref{max2} is still valid if instead of the outward normal direction at $P_0$ another outward direction is considered apart from the tangential one.
\end{rem}

\subsection{Estimates near the boundary}\label{eb}
In order to tackle the difficulties arising from the presence of the non-local term $K(t)=\left(\int_{D} f(u(x, t))\,dx \right)^{-q},
$ in \eqref{eq1}-\eqref{eq3} we need to estimate the contribution of  $u(x,t)$ near the boundary. For that purpose we will use the moving plane method as in \cite{kavtza}, which is actually inspired by the results in the seminal paper by  Gidas et al. \cite{gnn}. Although most of the implemented arguments are quite standard in the context of deterministic PDEs, since it is the first time that those ideas are employed for SPDEs a detailed proof is provided.
\begin{lem}\label{moving}
Let $u(x,t)$ be the solution of \eqref{eq1}-\eqref{eq3} with initial data $\xi \in L^2(\Omega;L^{\infty}(D))$ satisfying $ \, \xi >0 $ a.s.   in $D$  with $\xi=0$ on $\partial D.$  Assume further that the nonlinearity $f$ is an increasing function as well as $D\subset \IR^d,\;d\geq 1,$ is convex and smooth enough as in Theorem \ref{max2}. Then there exists $\ol{D}_{0} \varsubsetneq D$ such that
\[
\int_{D}f(u) \, dx \leq (\ell+1) \int_{D_{0}} f(u) \, dx,\quad\mbox{for all}\quad 0<\widetilde{t}_0\leq  t <T,\quad\mbox{a.s.},
\]
where $\ell$ is some positive integer.
\end{lem}
\begin{proof}
For any $y\in \pl D$ we define the hyperplane
$$\mathcal{T}(\mu,y):=\left\{x\in \IR^d:(x,\nu(y))_{d}=\mu\right\},$$
where $(\cdot,\cdot)_d$ stands for the inner product in $\IR^d.$

Then we can find $\mu_0$ such that $\mathcal{T}(\mu_0,y)$ coincides with the tangent hyperplane to $D$ at $y$ and $y\in \mathcal{T}(\mu_0,y)\cap \overline{D}$ (note that when $D$ is strictly convex then $\mathcal{T}(\mu_0,y)\cap \overline{D}=\{y\}$), see Fig.~\ref{fig2}.

Since $D$ is a bounded set there exists $\mu_1<\mu_0$ such that $\mathcal{T}(\mu,y)\cap \overline{D}=\emptyset$ for $\mu>\mu_0$ and $\mu<\mu_0-\mu_1.$

We define
$$\Sigma(\mu,y):=\{x\in D: \mu< (x, \nu(y))_d <\mu_0 \},$$
while by $\Sigma^{\,'}(\mu,y)$ we denote the reflection of $\Sigma(\mu,y)$ across $\mathcal{T}(\mu,y).$ Now using the convexity of $D$  we can choose $\bar{\mu}$ sufficiently close to $\mu_0$ so that $\Sigma^{\,'}(\bar{\mu},y)\subset D,$ see also Fig.~\ref{fig2}.

Applying  now Theorem \ref{max2}, since all its hypotheses are satisfied (see also Theorem \ref{t4}),  we deduce that for any $y\in \pl D$
$$\frac{\pl u(y,t)}{\pl \nu}=(\nabla u(y,t), \nu(y))_d<0,\quad\mbox{a.s.,}\quad\mbox{for any}\quad t\geq t_0>0.$$
By the spatial regularity of $u,$  see \eqref{reg}, we can find a neighbourhood of $y,$ say $\mathcal{N}_y,$ such that
$$\frac{\pl u(x, t_0)}{\pl \nu}=(\nabla u(x,t_0), \nu(y))_d<0,\quad\mbox{a.s.,}\quad\mbox{for any}\quad x\in \mathcal{N}_y.$$
We consider now a coordinate system centered at $y$ and defined by  $(y;\nu(y), \mathcal{T}(\mu_0,y))$ such that every $x\in \IR^d$ is expressed as $x=(x_{\nu},x_{\mathcal{T}}),$ where $x_{\nu}$ is the component in the direction of $\nu(y)$ while $x_{\mathcal{T}}$ stands for the component in the direction of the hyperplane $\mathcal{T}(\mu_0,y).$

Let us define the cylinder
$C_\de(y)=\{y=(x_{\nu}, x_{\mathcal{T}})\in \IR^d\big |\,|x_{\nu}|<\de,\;|x_{\mathcal{T}}|<\de \}.$
We may pick $\delta>0$ small enough so that the reflection of $\ol{C_\de(y)}\cap D$ across $\mathcal{T}(\bar{\mu},y),$ denoted by $C_\de^{\;'}(y),$ is compact in $D.$

Set $K_y=\mathcal{T}(\mu_0,y)\cap \overline{D},$ then $K_{y}$ is a compact convex set and
$K_y=\bigcap_{\mu<\mu_0} \Sigma(\mu,y).$
Every $\hat{y}\in K_{y}$ has the same exterior normal $\nu(y).$ Then we can define an open neighbourhood of $\hat{y}$ of the shape $C_{\de}(\hat{y})$ and on which $(\nabla u(\hat{y},t_0),\nu(y))_d<0$ almost surely (a.s.).  Moreover, $K_y\subset \bigcup_{\hat{y}} C_{\de}(\hat{y})$
and since $K_y$ is compact we can extract a finite cover of $C_{\de}(\hat{y}),$ say
$E=\bigcup_{i=1}^{n}C_{\de}(\hat{y_i})$
which contains $K_y,$ for some positive integer $n=n(y).$

\begin{figure}[b]
\begin{center}
\includegraphics[scale=1.2]{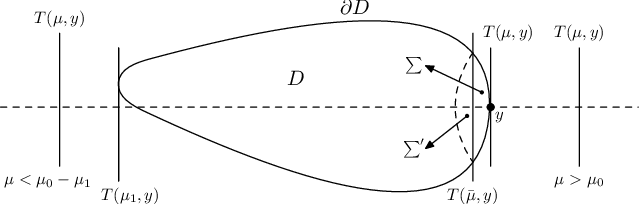}
\end{center}

\caption{The moving plane parallel configuration}
\label{fig2}       
\end{figure}

Since $D$ is convex we can find $\mu<\mu_0$ such that $\Sigma (\mu,y)\subset E$ and $\Sigma^{\,'}(\rho_0,y)\subset D,\; \Sigma(\rho_0,y)\cup\Sigma^{\,'}(\rho_0,y)\subset E$ for $\rho_0=\frac{\mu+\mu_0}{2}.$ (Note that if $D$ is strictly convex then the above construction is unnecessary).

We now set $z(x,t)=z(x_{\nu},x_{\mathcal{T}},t)=u(2\rho_0-x_{\nu},x_{\mathcal{T}},t)$ for $x\in \Sigma(\rho_0,y);$ actually $z$ is the reflection of $u$ across $T(\rho_0,y).$ Then $z$ is a weak solution of
\bgee
&& \frac{\pl z}{\pl t}  = \Delta z +K(t) f(z)+ \sigma (z) \partial_{t} W(x,t),\quad\mbox{on}\quad\Sigma(\rho_0,y)\times (\widetilde{t}_0,T_{max})\\
&& z\geq u\geq 0\quad\mbox{on}\quad  K_1:=\left(\pl D\cap \Sigma(\rho_0,y)\right)\times(\widetilde{t}_0,T_{max}),\\
&& z=u\quad \mbox{on}\quad  K_2:=\left( D\cap \mathcal{T}(\rho_0,y)\right)\times(\widetilde{t}_0,T_{max}).
\egee
Consequently $z$ and $u$ satisfy in a weak form the same SPDE on $\Sigma(\rho_0,y)\times (\widetilde{t}_0,T_{max})$ while
$z\geq u$ on $K_1\cup K_2$ and $z(\cdot,t_0)\geq u(\cdot,t_0)$ on $\Sigma(\rho_0,y)$ almost surely (a.s.),  hence by the comparison principle, \cite[Section 5.1]{cpt}, we deduce that $z\geq u$ almost surely (a.s.) on $\Sigma(\rho_0,y)\times (\widetilde{t}_0,T_{max}).$

Note that $\Sigma(\rho_0,y)$ contains an open set of the type $C_{\de}(y)\cap D$ and if we choose $\de<\mu_0-\rho_0$ then the reflection of $C_{\de}(y)\cap D$ across $\mathcal{T}(\rho_0,y)$ has a compact closure in $D.$ We can repeat the above construction for any $y\in \pl D$ and the collection of all cylinders $\{C_{\de}(y)\}_{y\in \pl D}$ builds up an open cover of $\pl D$ from which we can extract a finite subcover denoted by $C_{\de}(y_1),...,C_{\de}(y_{\ell})$ such that $\pl D\subseteq C_{\de}(y_1)\cup...\cup C_{\de}(y_{\ell}).$

Set $D_0=\displaystyle{D\setminus\bigcup_{i=1}^\ell C_{\de}(y_i)},$ then $\overline{D}_0\subset D$ and we have
\bgee
\int_{D} u \, dx
&&\leq
\int_{D_{0}} u \, dx
+
\sum_{i=1}^{\ell} \int_{C_{\de}(y_{i})\cap D} u\, dx  \leq \int_{D_{0}} u \, dx
+
\sum_{i=1}^{\ell} \int_{C_{\de}(y_{i})\cap D} z \, dx \notag \\
&& \leq
 \int_{D_{0}} u \, dx
+
\sum_{i=1}^{\ell} \int_{C^{\prime}_{\de}(y_{i})} z \, dx  =
 \int_{D_{0}} u \, dx
+
\sum_{i=1}^{\ell} \int_{C^{\prime}_{\de}(y_{i})} u \, dx \notag \\
&& \leq
\int_{D_{0}} u \, dx
+ \ell \int_{D_{0}} u \, dx \leq (\ell+1) \int_{D_{0}} u \, dx\quad\mbox{a.s.},  \notag
\egee
taking also into account that $u\leq z\; \mbox{on} \; C_{\de}(y_{i})\cap D\;$ and $\;u= z \quad \mbox{on} \; C^{\prime}_{\de}(y_{i})$  a.s.  by reflection.

Now since $f(s)$ is increasing we finally  deduce
\[
\int_{D}f(u) \, dx \leq (\ell+1) \int_{D_{0}} f(u) \, dx\quad\mbox{a.s.},
\]
and the proof of lemma  is now complete.
\end{proof}
\subsection{Finite-time blow-up}\label{bnlt}
Henceforth, the nonlinearity $f(s)$ is imposed to satisfy
\be\label{cond2}
[f^{1-q}(s)]''\geq 0\quad\mbox{for}\quad s\in \IR \quad\mbox{and}\quad\int_{b}^{\infty} \frac{ds}{f^{1-q}(s)}  < \infty, \quad\mbox{for any} \, b\in \IR.
\ee
We first prove a blow-up result when the parameter $\la$ is large enough.
\begin{thm}\label{tbu1}
Suppose that  \eqref{eq1}-\eqref{eq3} has
a (unique) local-in-time solution $u$ whose existence is provided by Theorem \ref{existence} . Assume further that the nonlinearity $f$ satisfies conditions \eqref{cond1} and \eqref{cond2} as well as $D\subset \IR^d,\;d\geq 1,$ is convex and smooth enough as in Theorem \ref{max2}.  Then $u$ blows up in finite time for sufficiently large values of the parameter $\lambda$, provided that $\xi\in L^2(\Omega;H)$ with $\xi(x)> 0$  a.s.  in $D$  and  $\xi=0$ on $\partial D.$
\end{thm}
\begin{proof}
Let us define $\widehat{u}(t)$ as in the proof of  Theorem \ref{ntbua}. Now taking the expectation over \eqref{bu1} we have
\bge
\mathbb{E}[ \hat{u} (t)]&&=\mathbb{E}\Big[ \int_{D}   \xi (x)\,\phi_{1}(x)  \, dx\Big]   -\la_1 \mathbb{E}   \Big[\int_{0}^{t} \int_{D}  u (x, s) \, \phi_{1} (x) \, dx\, ds \Big]\no \\
\qquad && +
\la\,\mathbb{E}\Big[ \int_{0}^{t} \int_{D} \frac{f(u(x, s))\, \phi_{1}(x)}{\left(\int_{D} f(u(x, s))\,dx \right)^{q}} \, dx \, ds\Big]\label{bu2}
\ege
taking also into account that
\[
\mathbb{E}\Big[\int_{0}^{t} \int_{D}  \sigma(u(x, s))\, \phi_{1}(x)\,dx \, dW(s)\Big] =0.
\]
For simplicity,  hereafter, we will write $u(x, t)$ and $\phi (x)$ as $u$ and $\phi$, respectively in the integrand.

Set $\Psi(t) = \mathbb{E}[ \hat{u} (t)],$ then by using again  Fubini's theorem, we deduce
\bge
\qquad  \Psi(t)=\Psi_0-\la_1\int_0^t \Psi(s)\,ds
+\la\,\mathbb{E}\Big[ \int_{0}^{t} \int_{D}\frac{f(u)\, \phi_{1}}{\left(\int_{D} f(u)\,dx \right)^{q}} \, dx \, ds\Big],\qquad\label{bu3}
\ege
where $ \Psi_0= \mathbb{E}\left[\left(\xi, \phi_{1}\right)_{H}\right],$ or equivalently the initial value problem
\bge
\frac{d\Psi}{dt} = - \lambda_{1} \Psi(t) +\la\, \mathbb{E}\Big[K(t)\int_{D} f(u)\, \phi_{1}\, dx \Big], \;\;t>0,\;\; \Psi(0)=\Psi_0.\qquad \label{e4_1}
\ege
By Lemma \ref{moving}, we can construct $D_{0}\subset D$ with $\overline{D}_{0} \varsubsetneq D$ such that
\[
\int_{D} f(u) \, dx \leq (\ell+1) \int_{D_{0}} f(u) \, dx,\quad\mbox{almost surely (a.s.)},
\]
for some $\ell \in \IN$. Let $m_1 = \inf_{ x \in D_{0}} \phi_{1}(x)$, then  since $\overline{D}_{0} \varsubsetneq D$  we have $m_1 >0$. Hence
\[
\int_{D} f(u) \, dx \leq \frac{\ell+1}{m} \int_{D_{0}} f(u)\, \phi_{1} \, dx \leq \frac{\ell+1}{m_1} \int_{D} f(u)\, \phi_{1}\, dx,\quad\mbox{almost surely (a.s.)},
\]
and so
\begin{equation}  \label{e4_2}
K(t)=\left( \int_{D} f(u) \, dx \right )^{-q} \geq L \left ( \int_{D} f(u)\, \phi_{1} \, dx \right )^{-q},\quad\mbox{almost surely (a.s.)},
\end{equation}
for
\be\label{con}
L= \left( \frac{m_1}{\ell+1} \right)^{q}.
\ee
Therefore by virtue of \eqref{e4_2} and  applying Jensen's inequality twice, since both $f(s)$ and $f^{1-q}(s)$ are convex functions, see also \eqref{cond1} and \eqref{cond2},  we deduce
\bge
\mathbb{E}\Big[K(t)\int_{D} f(u)\, \phi_{1} \, dx \Big]
&&\geq \mathbb{E}\left[L\,\left ( \int_{D} f(u)\, \phi_{1} \, dx \right )^{1-q}\right]\no\\
&&\geq L\,f^{1-q}\left(\mathbb{E}[\hat{u}(t)]\right)  =L f^{1-q}\left(\Psi(t)\right)\label{bu4}.
\ege
Thus by virtue of
\eqref{e4_1} and \eqref{bu4}  the differential inequality holds
\bgee \label{e4_3}
\frac{d\Psi(t)}{dt} \geq - \lambda_{1} \Psi(t) + \lambda L f^{1-q}\left(\Psi(t)\right),\;\;t>0,
\egee
with initial condition $\Psi(0)=\Psi_0.$

Define
$$0 < N := \sup_{s > \Psi(0)} \frac{s}{f^{1-q}(s)},$$
then due to \eqref{cond2} we have that $N<\infty,$ and so choosing  $\lambda >\frac{\la_1 N}{L},$ we deduce
\begin{align}
t
&  \leq
\int_{\Psi(0)}^{\Psi(t)} \frac{ds} {\lambda L f^{1-q}(s) - \lambda_{1} s}
\leq \frac{1}{\Lambda} \int_{\Psi(0)}^{\Psi(t)} \frac{ds}{f^{1-q}(s)}  < \frac{1}{\Lambda} \int_{\Psi(0)}^{\infty} \frac{ds}{f^{1-q}(s)} < \infty, \notag
\end{align}
for
\be\label{est}
0 < \Lambda \leq \lambda L- \la_1 N < \infty.
\ee
Thus $\Psi(t)$ blows up in finite time, i.e.
 $
\Psi(t) \to \infty
$ as $ t \to T_1$ where $T_1$  is estimated as
\be\label{est2}
T_1\leq \int_{\Psi(0)}^{\infty}  \frac{ds}{\lambda L f^{1-q}(s) - \lambda_{1} s} \leq \frac{1}{\Lambda} \int_{\Psi(0)}^{\infty} \frac{ds}{f^{1-q}(s)} < \infty.
\ee
Indeed,  since by Theorem \ref{t4} $u$ is bounded in $D\times [0,T),$  then  \eqref{nor} yields
\[
\Psi(t) = \mathbb{E}\Big[ \int_{D} u\, \phi_{1}(x) \, dx\Big] \leq \mathbb{E}\big[\| u \|_{\infty}\big],
\]
and thus
$\mathbb{E}\big[\| u \|_{\infty}\big] \to \infty$ as $t \to T^{*} \leq T_1.$
\end{proof}
Next  we prove that  blow-up also occurs for large enough initial data.
\begin{thm}\label{tbu2}
Suppose that  the assumptions of Theorem \ref{tbu1} hold true.
Assume also that
\be
\mathbb{E}\left[\int_{D} \xi\,\phi_1\,dx\right]>\zeta,
\ee
where $\zeta=\zeta(\la)$ is the largest root of the equation $$\alpha(s):=\la\,L\, f^{1-q}(s)-\la_1 s=0,$$ and $L$ is the constant given by \eqref{con}.
Then the solution $u$ of \eqref{intro1}-\eqref{intro2}  blows up in finite time.
\end{thm}
\begin{proof}
Following the same steps as in the proof of Theorem \ref{tbu1} we obtain that $\Psi(t) = \mathbb{E}\Big[ \int_{D} u\, \phi_{1}\, dx\Big]$ satisfies the differential inequality
\bgee \label{e43a}
\frac{d\Psi(t)}{dt} \geq - \lambda_{1} \Psi(t) + \lambda L f^{1-q}\left(\Psi(t)\right)=\alpha(\Psi(t)),\;\;t>0,
\egee
with $\Psi(0)=\Psi_0:=\mathbb{E}\left[\int_{D} \xi\,\phi_1\,dx\right].$

Let $\zeta=\zeta(\la)$ be the largest root of the equation $\alpha(s)=0.$  Then by choosing $\Psi_0>\zeta$ and using again \eqref{cond2} we deduce
$$
 t\leq \int_{\Psi_0}^{\Psi(t)} \frac{ds}{\alpha(s)} \leq \int_{\Psi_0}^{\infty} \frac{ds}{\alpha(s)}\leq \frac{1}{\Lambda_1} \int_{\Psi_0}^{\infty} \frac{ds}{f^{1-q}(s)}<\infty,
$$
for some positive constant $\Lambda_1.$ But the above relation entails  that $\Psi(t)$ blows up in finite time $T_1<\infty,$ where
$$
T_1 \leq \frac{1}{\Lambda_1} \int_{\Psi_0}^{\infty} \frac{ds}{f^{1-q}(s)}<\infty,
$$
which, similarly to Theorem \ref{tbu1}, implies that $\mathbb{E}\big[\| u \|_{\infty}\big] \to \infty$ as $t \to T^{*} \leq T_1.$
\end{proof}
\begin{rem}
Theorems \ref{tbu1} and \ref{tbu2} both imply explosion in terms of the expectation of the  $L^q-$norm for any $q\geq 1$ as well. Indeed, since $\phi_1$ is bounded and continuous on $\ol{D}$ by applying H\"{o}lder's inequality for each $q\geq 1$ we derive
\bgee
\Psi(t)\leq C_q\, \mathbb{E}\left[ \left(\int_{D} |u|^q \, dx\right)^{1/q}\right],
\egee
for $C_q=\left(\int_{D} |\phi_1|^r \, dx\right)^{1/r}$ with $r=q/(q-1),$ which actually yields that
the expectation of the $L^q-$norm explodes in finite time $T_q\leq T^{*}.$
\end{rem}
\subsection{An estimate of the probability of blow-up}\label{ebt}
In the current subsection  we consider the following
\bge
&&\frac{\pl u}{\pl t}=\Delta u+F(u)+\kappa u\; d \beta(t),\; (x,t)\in D_T,\quad\label{seq1}\\
&& u(x,t)=0,\quad (x,t)\in \Gamma_T,\quad\label{seq2}\\
&& u(x,0)=\xi(x),\quad x\in D\label{seq3},
\ege
where now $\{\beta(t),\;t\geq 0\}$ stands for a standard one-dimensional Brownian motion and $\kappa$ is positive constant. Now, for sake of simplicity we fix the parameter $\lambda=1$ and thus
\bgee
 F(u)= \frac{f(u)}{\Big ( \int_{D} f(u) \, dx \Big )^{q}}, \quad 0<q<1.
\egee
The domain  $D\subset \IR^d,\; d\geq 1,$ is still assumed to be convex as well as $\Gamma_T$ is smooth enough so that it has the inside strong sphere property whereas the nonlinearity $f(s)$ satisfies \eqref{cond1} and \eqref{cond2}. Thus,  it is easily seen that the  above problem satisfies the assumptions  of Lemma \ref{moving} and thus estimate \eqref{e4_2} is still valid for its solution.

Next we show that the solution $u$ of \eqref{seq1}-\eqref{seq3} exhibits a finite-time blow-up, induced again by the non-local term, in the sense
$$\limsup_{t\to T} ||u(\cdot,t)||_{\infty}=\infty,$$
and thus under these circumstances a  stronger type rather than blow-up in mean $L^p-$norm takes place.

For that purpose we employ a different technique than the one in subsection \ref{bnlt},which also  provides an upper estimate of the probability of blow-up. To this end, we first introduce the auxiliary random function
\bge\label{mk1}
v(x,t)=e^{-\kappa \beta(t)}\, u(x,t),\quad (x,t)\in D_T,
\ege
and we follow closely the approach introduced in \cite{dlm10}. In order, to make  our paper  self-contained, we present all the required steps in every detail.

By virtue of It\^{o}'s Lemma, see \cite[Proposition 1]{dlm10}, it can be shown that $v(x,t)$ satisfies the following random PDE problem
\bge
&&\frac{\pl v}{\pl t}=\Delta v-\frac{\kappa^2}{2}v+e^{-\kappa \beta(t)}F\left(e^{\kappa \beta(t)}v\right),\quad (x,t)\in D_T,\quad\label{req1}\\
&& v(x,t)=0,\quad (x,t)\in \Gamma_T,\quad\label{req2}\\
&& v(x,0)=\xi(x),\quad x\in D\label{req3}.
\ege
Notably,  \eqref{req1}-\eqref{req3} should be understood trajectorywise and classical results such as existence, uniqueness and positivity of its a solution up to eventual blow-up can be found in \cite[Theorem 9, Chapter 7]{fri}. Note that the solution $u$ of \eqref{seq1}-\eqref{seq3} blows up in finite time as long as the solution $v$ of \eqref{req1}-\eqref{req3} does so and due to \eqref{mk1} both of them blow up simultaneously.

Set
\bgee
 (u(t),\phi_1):=\int_{D} u(x,t)\, \phi_{1}(x) \, dx,
\egee
where again $\phi_1$ stands for the first Dirichlet eigenfunction of $-\Delta$ with corresponding eigenvalue $\lambda_1>0,$ then by Definition \ref{weak} we have
\bge\label{gv2}
(u(t), \phi_1)
=
(\xi, \phi_1) + \int_{0}^{t} \Big [
(u(s), \Delta \phi_1) + (F(u(s)), \phi_1) \Big ] \, ds
+
\kappa\int_{0}^{t} \left( u(s), \phi_1 \right )\, d \beta(s).\quad
\ege
Next  by virtue of \eqref{mk1} and It\^o's formula  we derive 
\bgee
e^{-\kappa \beta(t)}=1-k\int_0^t e^{-\kappa \beta(s)}\, d\beta(s)+\frac{k^2}{2}\int_0^t  e^{-\kappa \beta(s)}\,ds
\egee
or equivalently in differential form
\bge\label{gav1}
d\left(e^{-\kappa \beta(t)}\right)=-ke^{-\kappa \beta(t)} d\beta(t)+\frac{k^2}{2} e^{-\kappa \beta(t)}
\ege
Applying now integration by parts formula, see \cite[Corollary 7.11, p. 119]{mac},
\bgee
\widetilde{v}(t):=(v, \phi_1)
&&=
(\xi, \phi_1) + \int_{0}^{t} e^{-\kappa \beta(s)}d (u(s), \phi_1) \\
&&+\int_{0}^{t} \left( u(s), \phi_1 \right )\,d\left(e^{-\kappa \beta(t)}\right)+ \left[e^{-\kappa \beta(t)}, (u(t),\phi_1) \right],
\egee
where the last term in the preceding relation is called quadratic variation and is defined as
\bgee
\left[e^{-\kappa \beta(t)}, (u(t),\phi_1)\right]:=-\int_0^t \kappa^2 e^{-\kappa \beta(s)} (u(s),\phi_1)\,ds,\quad\mbox{for}\quad 0<t<T,
\egee
see also \cite[Definition 7.6]{mac}.

Taking now into account \eqref{gv2} we finally get
\bgee
&&\widetilde{v}(t):=(v, \phi_1)=
(\xi, \phi_1) + \int_{0}^{t} e^{-\kappa \beta(s)}\Big [
(u(s), \Delta \phi_1) + (F(u(s)), \phi_1) \Big ] \, ds\\
&&+\kappa\int_0^t e^{-\kappa \beta(s)}(u(s), \phi_1)\,d\beta(s)+\int_{0}^{t} \left( u(s), \phi_1 \right )\, \left(-\kappa e^{-\kappa \beta(s)} d \beta(s)+\frac{\kappa^2}{2} e^{-\kappa \beta(s)}\,ds \right)\\
&&+ \left[e^{-\kappa \beta(t)}, (u(t),\phi_1) \right],
\egee
Consequently  by virtue of \eqref{mk1} we have
\bgee
(v(t), \phi_1)=
(\xi, \phi_1) + \int_{0}^{t} \left[
-\lambda_1(v(t), \phi_1) + e^{-\kappa \beta(s)}\left(F\left(e^{\kappa \beta(s)}v(s)\right), \phi_1\right)-\frac{\kappa^2}{2} (v(s),\phi_1) \right]\, ds
\egee
which by differentiation with respect to time gives
\bgee
\frac{d \widetilde{v}(t)}{dt}=-\left(\lambda_1+\frac{\kappa^2}{2}\right) \widetilde{v}(t)+ e^{-\kappa \beta (t)}\left(F\left(e^{\kappa \beta(t)}v(t)\right),\phi_1\right),
\egee
and by virtue of \eqref{e4_2} entails
\bge\label{ps1}
\frac{d \widetilde{v}(t)}{dt}\geq
-\left(\lambda_1+\frac{\kappa^2}{2}\right) \widetilde{v}(t) +  L\, e^{-\kappa \beta(t)}\left(f\left(e^{\kappa \beta(t)}v(t)\right), \phi_1\right)^{1-q}.
\ege
Assuming now that the nonilearity $f(s)$ satisfies the growth condition
\bge\label{gc1}
f(s)\geq L^{\frac{1}{(q-1)}} s^{\frac{1+\varepsilon}{1-q}}\quad\mbox{for all}\quad s>0\quad\mbox{and some}\quad \varepsilon >0,
\ege
then using Jensen's inequality \eqref{ps1} we have
\bge\label{ps2}
\frac{d \widetilde{v}(t)}{dt}\geq
-\left(\lambda_1+\frac{\kappa^2}{2}\right) \widetilde{v} (t)+  e^{\varepsilon \kappa \beta(t)}( \widetilde{v}(t))^{1+\varepsilon}, \quad 0<t<T.
\ege
Comparing now the solution of \eqref{ps2} with the solution of the following Bernoulli's type initial value problem
\begin{equation*}
\frac{d \mathcal{Y}(t)}{dt}=-\left( \lambda_1 +\frac{\kappa ^{2}}{2}\right)
\mathcal{Y}(t)\,+ e^{\varepsilon\kappa \beta (t) } (\mathcal{Y}(t))^{1+\varepsilon }\,,\text{ \ }\mathcal{Y}(0)=(\xi,\phi_1):=\mathcal{\xi}_0,
\end{equation*}
which is given by
\begin{equation*}
\mathcal{Y}(t)\ = \ e^{-(\lambda_1 +\frac{\kappa^2}{2})t}\left[ \xi_1^{-\varepsilon }-\varepsilon
\int_{0}^{t}e^{-(\lambda_1 +\frac{\kappa^2}{2})\varepsilon s+\varepsilon\kappa \beta (s)}\,ds
\right] ^{-\frac{1}{\varepsilon }},\quad 0\leq t<T_{\mathcal{Y}} ,
\end{equation*}
we get that
\bge\label{compp}
\widetilde{v}(t)\geq \mathcal{Y}(t)\quad\mbox{for all}\quad 0\leq t\leq \min\{T,T_{\mathcal{Y}} \},
\ege
where
\begin{equation}  \label{tautau}
T_{\mathcal{Y}} :=\inf \left\{ t\geq 0\left|\ \int_{0}^{t}e^{-(\lambda_1 +\frac{\kappa^2}{2})\varepsilon s+\varepsilon\kappa \beta (s)}\,ds \geq \frac{1}{\varepsilon }\xi_1^{-\varepsilon }\right.\right\},
\end{equation}
denotes the maximum existence time of $\mathcal{Y}(t).$ Note that $\mathcal{Y}(t)$ exhibits finite-time blow-up in the event $\{T_{\mathcal{Y}}<\infty\}$ and due to \eqref{compp} the function
\bgee
t\longmapsto\int_D u(x,t)\phi_1(x)\,dx,
\egee
explodes in finite time on the event  $\{T_{\mathcal{Y}}<\infty\}.$ Furthermore, $T_{\mathcal{Y}}$ is an upper bound of the blow-up time of $\widetilde{v}(t)$ and since
\bgee
\widetilde{v}(t)=\int_{D} v(x,t)\,\phi_1(x)\,dx\leq ||v(\cdot,t)||_{\infty},
\egee
it is also an upper bound of blow-up times for $v$ and $u.$ We are now ready to provide a lower bound of the probability of blow-up  for $v$ and $u.$ First, by \eqref{tautau} we have
\bge
P[T_{\mathcal{Y}} =+\infty ] &=&P\left[ \int_{0}^{t}\exp \left(-\left(\lambda_1 +\frac{\kappa
^{2}}{2}\right)\varepsilon s+\kappa \varepsilon \beta(s)\right)\,ds<\frac{1}{\varepsilon }\xi_1^{-\varepsilon }
\text{ for all }t>0\right]  \no\\
&=&P\left[\int_{0}^{\infty}\exp \left(-\left(\lambda_1 +\frac{\kappa
^{2}}{2}\right)\varepsilon s+\kappa \varepsilon  \beta(s) \right)\,ds\leq \frac{1}{\varepsilon }\xi_1^{-\varepsilon }\right]\no \\
&=&P\left[ \int_{0}^{\infty }\exp (2\hat{\varepsilon } \beta(s)^{(\ell )})\,ds\leq
\frac{1}{\varepsilon }\xi_1^{-\varepsilon }\right] \label{mk2},
\ege
where $ \beta(s)^{(\ell )}:=\ell s+ \beta(s)$, $\ell :=-\frac{(\lambda_1 +\frac{\kappa
^{2}}{2})} {\kappa}$, and $\hat{\varepsilon}:=\frac{\kappa \varepsilon} {2}$. Using the new time scale $s\mapsto s\cdot \hat{\varepsilon}^2$ we finally get
\begin{equation}  \label{doubt}
P[T_{\mathcal{Y}}=+\infty ]=P\left[ \frac{4}{\kappa ^{2}\varepsilon ^{2}}\int_{0}^{\infty
}\exp (2\beta(s)^{(\widehat{\ell })})\,ds\leq\frac{1}{\varepsilon }\xi_1^{-\varepsilon }\right] ,
\end{equation}
where $\hat{\ell}=\frac{\ell}{\hat{\varepsilon}}.$ The distribution of the integral term in \eqref{doubt} can be identified by either using some formulas in \cite{BS02, Duf, Yor} or otherwise by following the approach in \cite{pp} and therefore we obtain
\begin{equation*}
\int_{0}^{\infty }\exp (2 \beta(s)^{(\hat{\ell })})\,ds\ =\ \frac{1}{ 2Z_{-
\hat{\ell }}},
\end{equation*}
where the above relation should be understood in distributional sense. Here $Z_{-\hat{\ell }}$ is a random variable following the law  $$P(Z_{-\hat{\ell }}\in dy)=\frac{1}{
\Gamma (-\hat{\ell }) }e^{-y}y^{-\hat{\ell }-1}dy,$$
where $\Gamma (\cdot)$ stands for the standard $\Gamma-$ function, cf. \cite{AS}. 

Consequently by \eqref{doubt}
\begin{equation*}
P[T_{\mathcal{Y}}=+\infty  ]=\int_{0}^{\frac{1}{\varepsilon }\xi_1^{-\varepsilon }}h(y)dy,
\end{equation*}
where
\begin{equation*}
h(y)\ =\ \frac{(\kappa ^{2}\varepsilon ^{2}y/2)^{(2\lambda_1 +\kappa ^{2})/\kappa
^{2}\varepsilon }}{y\Gamma ((2\lambda_1 +\kappa ^{2})/(\kappa ^{2}\varepsilon ))}\exp
\left( -\frac{2}{\kappa ^{2}\varepsilon ^{2}y}\right),
\end{equation*}
see also  \cite[formula 1.10.4(1)]{BS02}, and thus
\bge\label{ik1}
P[T_{\mathcal{Y}}<+\infty  ]=1-P[T_{\mathcal{Y}}=+\infty  ]=\int_{\frac{1}{\varepsilon }\xi_1^{-\varepsilon
}}^{+\infty }h(y)\,dy.
\ege
In this manner we have shown the following
\begin{thm}
Let the domain $D\in \IR^d,\; d\geq 1,$ be convex and smooth enough as in Theorem \ref{max2}. Assume further that the nonlinearity $f(s)$ satisfies the growth condition \eqref{gc1}. Then the probability that the solution of \eqref{seq1}-\eqref{seq3} exhibits finite-time blow-up is lower bounded by the quantity $\int_{\frac{1}{\varepsilon }\xi_1^{-\varepsilon
}}^{+\infty }h(y)\,dy.$
\end{thm}
\begin{rem}
Note that relation \eqref{ik1} guarantees that the solution of \eqref{seq1}-\eqref{seq3} blows up almost surely once a big enough initial datum $\xi(x)$ is considered, i.e. once $$\xi_1=\int_{D} \xi(x)\,\phi_1(x)\gg 1,$$
which is in agreement with the result of Theorem \ref{tbu2}.
\end{rem}
\begin{rem}In case we consider a non-local term of the form
\bgee
 F(u)= \frac{\lambda f(u)}{\Big ( \int_{D} f(u) \, dx \Big )^{q}}, \quad 0<q<1,
\egee
then
\bgee
T_{\mathcal{Y}} :=\inf \left\{ t\geq 0\left|\ \int_{0}^{t}e^{-(\lambda_1 +\frac{\kappa^2}{2})\varepsilon s+\varepsilon\kappa W_{s}}\,ds \geq \frac{1}{\lambda \varepsilon }\xi_1^{-\varepsilon }\right.\right\}
\egee
and thus
\bgee
P[T_{\mathcal{Y}}<+\infty  ]=\int_{\frac{1}{\lambda \varepsilon }\xi_1^{-\varepsilon
}}^{+\infty }h(y)\,dy.
\egee
Therefore the bigger the value of the control parameter $\lambda$ is then the more probable the solution of \eqref{seq1}-\eqref{seq3} to exhibit finite-time blow-up becomes; the latter is also consistent with the result of Theorem \ref{tbu1}.
\end{rem}
\begin{rem}
Note that for $\kappa=0,$ when problem \eqref{seq1}-\eqref{seq3} becomes deterministic and thus $u=v,$ then by \eqref{mk2} we derive that $P[T_{\mathcal{Y}}=+\infty  ]=0$ provided that $\xi_1>\lambda_1^{1/\varepsilon},$ recovering the probabilistic counterpart of \cite[Theorem 2]{kavtza}.
\end{rem}
\begin{rem}
A global-in-time existence result for problem \eqref{seq1}-\eqref{seq3} can be derived following the same lines as in \cite[Theorem 5]{dlm10}  once the non-linearity $f(s)$ is strictly positive, increasing  and satisfies a growth condition of the form
\bgee
f(s)\leq C s^{1+\varepsilon}\quad\mbox{for all}\quad s>0\quad\mbox{and some}\quad C,\varepsilon >0,
\egee
since then
\bgee
 F(u)= \frac{f(u)}{\Big ( \int_{D} f(u) \, dx \Big )^{q}}\leq \widetilde{C} u^{1+\varepsilon}\quad\mbox{for all}\quad u>0,
\egee
for $\widetilde{C}=\frac{C}{(|D| f(0))^q}.$ In that case, a lower bound of the maximum existence time $T>0$ for the solution $u$ of \eqref{seq1}-\eqref{seq3} can be also derived, see for example \cite[Theorem 5]{dlm10}.
\end{rem}
\subsection*{Acknowledgement} The authors would like to thank the anonymous referees for their valuable comments which improved substantially the current manuscript. They would like also to thank Dr. M. Hofmanov\'{a} for letting them  know about the paper \cite{debdehof} as well as   Prof. S. Larsson for his stimulating comments which helped improving the presentation of some of the presented results. Finally, special thanks go to Dr. Joe Gildea for helping the authors with the construction of  Fig .~\ref{fig1} and  Fig.~\ref{fig2}.



\end{document}